\newif\ifpreprint
\preprinttrue  %

\ifpreprint

\documentclass[11pt]{article}
\usepackage{fullpage}
\title{Learning Acceleration Algorithms for Fast Parametric \\Convex Optimization with Certified Robustness}

\author{Rajiv Sambharya, Jinho Bok, Nikolai Matni, and George Pappas}
\date{%
    University of Pennsylvania\\
    \today
}

\usepackage[round,authoryear]{natbib}
\usepackage{amsmath,enumitem,mathtools,amsthm,amssymb}

\renewcommand\arraystretch{1.2}

\newtheorem{theorem}{Theorem}

\newtheorem{lemma}[theorem]{Lemma}
\newtheorem{example}[theorem]{Example}
\newtheorem{proposition}[theorem]{Proposition}

\newtheorem{assumption}[theorem]{Assumption}

\newtheorem{definition}{Definition}[section]

\usepackage[colorlinks,
            linkcolor=cyan,
            citecolor=cyan,
            urlcolor=cyan,
            bookmarks]{hyperref}

\else %

\documentclass[final,onefignum,onetabnum]{siamart250211}

\usepackage{lipsum}
\usepackage{amsfonts}
\usepackage{graphicx}
\usepackage{epstopdf}
\usepackage{bibentry}
\ifpdf
  \DeclareGraphicsExtensions{.eps,.pdf,.png,.jpg}
\else
  \DeclareGraphicsExtensions{.eps}
\fi

\usepackage{enumitem}
\setlist[enumerate]{leftmargin=.5in}
\setlist[itemize]{leftmargin=.5in}

\newsiamremark{remark}{Remark}
\newsiamremark{hypothesis}{Hypothesis}
\crefname{hypothesis}{Hypothesis}{Hypotheses}
\newsiamthm{claim}{Claim}

\headers{Learning Robust Acceleration Methods for Convex Optimization}{R. Sambharya, J. Bok, N. Matni, and G. Pappas}

\title{Learning Acceleration Algorithms for Fast Parametric Convex Optimization with Certified Robustness\thanks{Submitted to the editors 08/13/2025.}}

\author{Rajiv Sambharya\thanks{University of Pennsylvania, Philadelphia, PA (\email{sambhar9@seas.upenn.edu}, \\ \email{jinhobok@wharton.upenn.edu}, \email{nmatni@seas.upenn.edu}, \email{pappasg@seas.upenn.edu})} \and Jinho Bok\footnotemark[2] \and Nikolai Matni\footnotemark[2] \and George Pappas\footnotemark[2]}

\usepackage{amsopn}

\let\citep\cite
\let\citet\cite

\usepackage{amssymb}
\usepackage{cite}
\newtheorem{example}[theorem]{Example}
\newtheorem{assumption}[theorem]{Assumption}

\renewcommand\arraystretch{1.2}

\usepackage[skip=8pt]{caption}
\usepackage{graphicx}
\usepackage{subcaption}
\usepackage{hyperref}
\fi
\usepackage{tikz}
\usetikzlibrary{positioning}
\usepackage{verbatim}
\usepackage{multicol}
\usepackage{algorithm}%
\usepackage[noend]{algpseudocode} %
\usepackage{bm}

\usepackage{subcaption}
\usepackage{makecell}
\usepackage{csvsimple}	%
\usepackage{booktabs}   %
\usepackage[flushleft]{threeparttable}  %
\usepackage[round-mode=places, round-integer-to-decimal, round-precision=4,
    table-format = 1.4,
    table-number-alignment=center,
    round-integer-to-decimal]{siunitx}
\usepackage{adjustbox}  %

\newcommand{\bnote}[1]{}
\renewcommand{\bnote}[1]{\textcolor{red}{\textbf{B. #1}}}

\usepackage{etoolbox}

\newcommand{\fplen}{n}

\DeclareMathOperator*{\argmin}{argmin}

\newcommand*{\startlegend}{-0.2}
\newcommand*{\enlegend}{0.7}

\DeclareDocumentCommand{\T}{ O{z} O{} }{T\IfValueT{#2}{(#1,\theta_{#2})}\IfNoValueT{#2}{(#1,\theta)}}
\DeclareDocumentCommand{\Tj}{ O{k} O{z} O{} }{T^{#1}\IfValueT{#3}_{\theta_{#3}}{(#2)}}

\DeclareDocumentCommand{\CB}{ O{} }{C_{B_{\IfValueTF{#1}{\theta_{#1}}{\theta}}}}
\DeclareDocumentCommand{\RB}{ O{} }{R_{B_{\IfValueTF{#1}{\theta_{#1}}{\theta}}}}

\newcommand{\eg}{{\it e.g.}}
\newcommand{\ie}{{\it i.e.}}

\newcommand{\ones}{\mathbf 1}
\newcommand{\reals}{{\mbox{\bf R}}}

\newcommand{\symm}{{\mbox{\bf S}}}  %

\newcommand{\Tr}{\mathop{\bf tr}}
\newcommand{\diag}{\mathop{\bf diag}}

\newcommand{\prox}{\textbf{prox}}

\newcommand{\Sec}{Section}
\newcommand{\Subsec}{Subsection}

\newcommand{\cblock}[3]{
  \hspace{-1.5mm}
  \begin{tikzpicture}
    [
    node/.style={square, minimum size=10mm, thick, line width=0pt},
    ]
    \node[fill={rgb,255:red,#1;green,#2;blue,#3}] () [] {};
  \end{tikzpicture}%
}

\newcommand{\ccircle}[3]{%
  \raisebox{0.55\height}{
  \begin{tikzpicture}[baseline=(node.base)]%
    \node[circle, fill={rgb,255:red,#1;green,#2;blue,#3}, yshift=14] (node) at (0,1) {};%
  \end{tikzpicture}%
  }
}

\newcommand{\linestraight}[3]{%
  \raisebox{0.5\height}{
  \begin{tikzpicture}%
    \draw[line width=2.5pt, color={rgb,255:red,#1;green,#2;blue,#3}] (0,-.1) -- (0.75,-.1);
  \end{tikzpicture}%
  }
}

\newcommand{\linediamond}[3]{%
  \begin{tikzpicture}[baseline={(0,-.1)}]
    \coordinate (A) at (0.15,0);
    \coordinate (B) at (.25,-.1);
    \coordinate (C) at (.35,0);
    \coordinate (D) at (.25,.1);
    \draw[draw=none, fill={rgb,255:red,#1;green,#2;blue,#3}] (A) -- (B) -- (C) -- (D) -- cycle;
    \draw[line width=0.8pt, color={rgb,255:red,#1;green,#2;blue,#3}] (\startlegend,0) -- (0.6,0);
  \end{tikzpicture}%
  }

\usetikzlibrary{shapes.geometric}

\newcommand{\linecircle}[3]{%
  \begin{tikzpicture}[baseline={(0,-.1)}]
    \draw[draw=none, fill={rgb,255:red,#1;green,#2;blue,#3}](0.25, 0) circle (.08);
    \draw[line width=0.8pt, color={rgb,255:red,#1;green,#2;blue,#3}] (\startlegend,0) -- (\enlegend,0);
  \end{tikzpicture}%
  }

\newcommand{\linesquare}[3]{%
  \begin{tikzpicture}[baseline={(0,-.1)}]
    \draw[draw=none, fill={rgb,255:red,#1;green,#2;blue,#3}](0.179, -.071) rectangle (.321,.071);
    \draw[line width=0.8pt, color={rgb,255:red,#1;green,#2;blue,#3}] (\startlegend,0) -- (\enlegend,0);
  \end{tikzpicture}%
  }

\newcommand{\linelefttri}[3]{%
  \begin{tikzpicture}[baseline={(0,-.1)}]
    \coordinate (A) at (0.17,0);
    \coordinate (B) at (.3,-.1);
    \coordinate (C) at (.3,.1);
    \draw[draw=none, fill={rgb,255:red,#1;green,#2;blue,#3}] (A) -- (B) -- (C) -- cycle;
    \draw[line width=0.8pt, color={rgb,255:red,#1;green,#2;blue,#3}] (\startlegend,0) -- (\enlegend,0);
  \end{tikzpicture}%
}

\newcommand{\linerighttri}[3]{
  \begin{tikzpicture}[baseline={(0,-.1)}]
    \coordinate (A) at (0.33,0);
    \coordinate (B) at (.2,-.1);
    \coordinate (C) at (.2,.1);
    \draw[draw=none, fill={rgb,255:red,#1;green,#2;blue,#3}] (A) -- (B) -- (C) -- cycle;
    \draw[line width=0.8pt, color={rgb,255:red,#1;green,#2;blue,#3}] (\startlegend,0) -- (\enlegend,0);
  \end{tikzpicture}%
}

\newcommand{\linedowntri}[3]{
  \begin{tikzpicture}[baseline={(0,-.1)}]
    \coordinate (A) at (0.25,-.09);
    \coordinate (B) at (.17,.06);
    \coordinate (C) at (.33,.06);
    \draw[draw=none, fill={rgb,255:red,#1;green,#2;blue,#3}] (A) -- (B) -- (C) -- cycle;
    \draw[line width=0.8pt, color={rgb,255:red,#1;green,#2;blue,#3}] (\startlegend,0) -- (\enlegend,0);
  \end{tikzpicture}%
}

\newcommand{\lineuptri}[3]{
  \begin{tikzpicture}[baseline={(0,-.1)}]
    \coordinate (A) at (0.25,.09);
    \coordinate (B) at (.17,-.06);
    \coordinate (C) at (.33,-.06);
    \draw[draw=none, fill={rgb,255:red,#1;green,#2;blue,#3}] (A) -- (B) -- (C) -- cycle;
    \draw[line width=0.8pt, color={rgb,255:red,#1;green,#2;blue,#3}] (\startlegend,0) -- (\enlegend,0);
  \end{tikzpicture}%
}

\newcommand{\rkfvisualslegend}{\vspace{-1mm} \small \\
\ccircle{253}{166}{167} Noisy measurements \quad
\ccircle{6}{165}{3} Optimal solution \\
\cblock{0}{0}{255} LAH Accel w/ robustness \quad
\cblock{255}{189}{8} LAH  \quad
\cblock{150}{150}{150} No learning
}

\newcommand{\legendlogistic}{\vspace{-1mm} \small 
\lineuptri{0}{0}{0} Nesterov \hspace{1mm}
  \linelefttri{166}{86}{40} nearest neighbor \hspace{1mm}
  \linecircle{228}{26}{28} LAH \hspace{1mm}
  \linesquare{55}{126}{184} LAH Accel
}

\newcommand{\legendrkf}{\vspace{-1mm} \small 
\lineuptri{0}{0}{0} cold start \hspace{1mm}
  \linelefttri{166}{86}{40} nearest neighbor \hspace{1mm} 
  \linedowntri{77}{175}{74} previous solution \hspace{1mm}\\
  \linediamond{153}{102}{153} LM \hspace{1mm}
  \linecircle{228}{26}{28} LAH \hspace{1mm}
  \linesquare{55}{126}{184} LAH Accel
}

\newcommand{\legendnonnegls}{\vspace{-1mm} \small 
\linedowntri{0}{0}{0} cold start \hspace{1mm}
  \linelefttri{166}{86}{40} nearest neighbor \hspace{1mm} 
  \linecircle{228}{26}{28} LAH \hspace{1mm}
  \linesquare{55}{126}{184} LAH Accel
}

\newcommand{\legendquad}{\vspace{-1mm} \small 
\linedowntri{0}{0}{0} cold start \hspace{1mm}
  \linelefttri{166}{86}{40} nearest neighbor \hspace{1mm} 
  \lineuptri{77}{175}{74} previous solution \hspace{1mm}
  \linesquare{55}{126}{184} LAH Accel
}

\newcommand{\legendlasso}{\vspace{-1mm} \small 
\lineuptri{0}{0}{0} Nesterov \hspace{1mm}
  \linelefttri{166}{86}{40} nearest neighbor \hspace{1mm} \\
  \linerighttri{152}{78}{163} LAH safeguarded \hspace{1mm}
  \linecircle{228}{26}{28} LAH no safeguard
  \linesquare{55}{126}{184} LAH Accel
}

\newcommand{\legendlogisticstep}{\vspace{-1mm} \small
  \linestraight{251}{124}{163} $2 / L$ for step sizes and $1$ for momentum sizes \hspace{2mm}
}

\newcommand{\iters}{\small{Mean iterations to reach a given primal and dual residual (Tol.)}}

\newcommand{\itersunconstrained}{\small{Mean iterations to reach a given suboptimality (Tol.)}}

\newenvironment{talign*}
 {\let\displaystyle\textstyle\csname align*\endcsname}
 {\endalign}

\newcommand{\tableheader}{
  \begin{tabular}{@{}c@{}}Tol. \\\end{tabular}
  &
\begin{tabular}{@{}c@{}}Nesterov\end{tabular}
&
\begin{tabular}{@{}c@{}}Nearest \\ Neighbor\end{tabular}
&
\begin{tabular}{@{}c@{}}L2WS \\ $N=10$\end{tabular}
&
\begin{tabular}{@{}c@{}}L2WS \\ $N=10000$\end{tabular}
&
\begin{tabular}{@{}c@{}}LM \\ $N=10$\end{tabular}
&
\begin{tabular}{@{}c@{}}LM \\ $N=10000$\end{tabular}
&
\begin{tabular}{@{}c@{}}LAH \end{tabular}
&
\begin{tabular}{@{}c@{}}LAH \\ Robust\end{tabular}
\\}

\newcommand{\colnames}{\colA & \colB & \colC &\colD &\colE & \colF & \colG & \colH & \colI}

\newcommand{\myCSVReaderRed}[1]{%
\tableheader
      \midrule
    \csvreader[
        head to column names,
        late after line=\\
    ]{#1}{ %
        accuracies=\colA,
        cold_start_red=\colC,
        nearest_neighbor_red=\colD,
        maml_red=\colO,
        obj_k0_red=\colE,
        obj_k5_red=\colF,
        obj_k15_red=\colG,
        obj_k30_red=\colH,
        obj_k60_red=\colI,
        reg_k0_red=\colJ,
        reg_k5_red=\colK,
        reg_k15_red=\colL,
        reg_k30_red=\colM,
        reg_k60_red=\colN,
    }{\colnames}
}

\newcommand{\myCSVReaderRedAlt}[2]{%
    \tableheaderAlt{#1}
      \midrule
    \ifstrequal{#1}{MAML}{%
    \csvreader[
        head to column names,
        late after line=\\
    ]{#2}{ %
        accuracies=\colA,
        cold_start_red=\colC,
        nearest_neighbor_red=\colD,
        maml_red=\colO,
        obj_k0_red=\colE,
        obj_k1_red=\colF,
        obj_k5_red=\colG,
        obj_k15_red=\colH,
        obj_k60_red=\colI,
        reg_k0_red=\colJ,
        reg_k1_red=\colK,
        reg_k5_red=\colL,
        reg_k15_red=\colM,
        reg_k60_red=\colN,
    }{\colnames}
    }
    {
      \csvreader[
        head to column names,
        late after line=\\
    ]{#2}{ %
        accuracies=\colA,
        cold_start_red=\colC,
        nearest_neighbor_red=\colD,
        prev_sol_red=\colO,
        obj_k0_red=\colE,
        obj_k1_red=\colF,
        obj_k5_red=\colG,
        obj_k15_red=\colH,
        obj_k60_red=\colI,
        reg_k0_red=\colJ,
        reg_k1_red=\colK,
        reg_k5_red=\colL,
        reg_k15_red=\colM,
        reg_k60_red=\colN,
    }{\colnamesprevsol}
    }
}

\newcommand{\useCSVReaderRed}[2]{%
    \ifdefined\currentCSVReaderRed
        \expandafter\csname\currentCSVReaderRed\endcsname{#1}{#2}%
    \else
        \PackageWarning{YourPackage}{No CSV reader command defined, defaulting to \myCSVReaderRed}%
        \myCSVReaderRed{#1}{#2}%
    \fi
}

\makeatletter
\renewcommand{\eqref}[1]{\textup{\tagform@{\ref{#1}}}}
\makeatother

\begin{document}
\maketitle

\begin{abstract}
    We develop a machine-learning framework to learn hyperparameter sequences for accelerated first-order methods (\eg, the step size and momentum sequences in accelerated gradient descent) to quickly solve parametric convex optimization problems with certified robustness.
    We obtain a strong form of robustness guarantee---certification of worst-case performance over all parameters within a set after a given number of iterations---through regularization-based training.
    The regularization term is derived from the performance estimation problem (PEP) framework based on semidefinite programming, in which the hyperparameters appear as problem data. 
    We show how to use gradient-based training to learn the hyperparameters for several first-order methods: accelerated versions of gradient descent, proximal gradient descent, and alternating direction method of multipliers.
    Through various numerical examples from signal processing, control, and statistics, we demonstrate that the quality of the solution can be dramatically improved within a budget of iterations, while also maintaining strong robustness  guarantees.
    Notably, our approach is highly data-efficient in that we only use ten training instances in all of the numerical examples.

\end{abstract}

\newcommand{\myparagraph}[1]{%
  \paragraph{#1\ifpreprint\unskip.\fi}%
}

\newcommand{\myparagraphstar}[1]{%
  \par\vspace{1ex}\noindent\textbf{#1\ifpreprint.\fi}\hspace{1em plus 0.5em}%
}

\ifpreprint \else
\begin{keywords}
Learning to optimize, Convex optimization, Acceleration, Performance estimation
\end{keywords}
\fi

\section{Introduction}\label{sec:intro}
In this work, we study \emph{parametric convex optimization problems}, where the goal is to repeatedly solve convex problems whose objective or constraints depend on a varying \emph{problem parameter}.
Such problems appear across many fields.
For instance, in model predictive control (MPC), we repeatedly solve an optimization problem at each sampling instance, treating the current state as the problem parameter that determines the next control input~\citep{borrelli_mpc_book}.
Similarly, in signal processing, we repeatedly solve a convex problem in which the new measurement vector is the problem parameter used to reconstruct the underlying signal while the measurement matrix remains fixed~\citep{lista}.
It turns out that these parametric convex optimization problems can typically be reformulated as \emph{parametric fixed-point problems} of the following form via their optimality conditions~\citep{mon_primer,osqp,scs_quadratic}:
\begin{equation}\label{prob:fp}
  \mbox{find} \; z \quad \mbox{ such that } \quad  z = T(z, x),
\end{equation}
where $z \in \reals^\fplen$ is the decision variable, $x \in \reals^d$ is the problem parameter, and $T : \reals^\fplen \times \reals^d \rightarrow \reals^\fplen$ is a known mapping.
First-order methods which only use gradients and subgradients are popular methods to solve convex optimization problems, including ones of the form~\eqref{prob:fp}~\citep{bauschke2011convex,fom_book, lscomo}.
Popular examples include gradient descent, proximal gradient descent~\citep{prox_algos}, and the alternating direction method of multipliers (ADMM)~\citep{dr_splitting,Boyd_admm}.
These methods typically take the form of \emph{fixed-point iterations} which repeatedly apply the operator $T$, giving the iterates
\begin{equation}\label{prob:fp_algo}
  z^{k+1}(x) = T(z^k(x),x).
\end{equation}
Under certain conditions on the operator $T$ (\eg, if $T$ is averaged or contractive), which hold for a wide variety of algorithms used in convex optimization, the iterates obtained through algorithm~\eqref{prob:fp_algo} are guaranteed to converge to an optimal solution; \ie, there exists a solution $z^\star(x)$ of~\eqref{prob:fp} such that $z^k(x) \to  z^\star(x)$~\citep[\Sec~2.4]{lscomo}.

Despite their popularity, first-order methods are known to suffer from slow convergence~\citep{anderson_acceleration,zhang2020globally}.
In many applications, we only have the time to run a finite number of iterations---\eg, in model predictive control, where each problem needs to be solved within milliseconds~\citep{borrelli_mpc_book}. 
Therefore, it is necessary to develop methods that can yield a high-quality solution within a limited amount of time.
On one hand, momentum-based \emph{acceleration methods} that combine past iterates to obtain the next one have been developed for faster convergence~\citep{accel_survey}.
A celebrated result in convex optimization is that the accelerated versions of gradient descent~\citep{nesterov} and of proximal gradient descent~\citep{fista} based on momentum provably improve the convergence rate over their unaccelerated versions. 
On the other hand, these acceleration methods are designed for a large function class (\eg, all convex and $L$-smooth functions), and are not tailored to the parametric problems of our interest.
Hence, it may still be possible to find algorithms that improve on general-purpose acceleration methods like Nesterov's method over the parametric family.
Moreover, for the more general case of fixed-point iterations, different acceleration schemes such as Anderson acceleration~\citep{anderson1965iterative,anderson_acceleration} are an active area of research despite their lack of general worst-case guarantees.

\emph{Learning to optimize} is a paradigm that leverages the parametric nature of problem~\eqref{prob:fp} to design tailored algorithms that solve such parametric problems quickly~\citep{l2o,amos_tutorial}. 
In this paradigm, the problem parameter $x$ is assumed to be drawn from a distribution $\mathcal{D}$, and the training dataset consists of samples $\{x_i\}_{i=1}^N$ where each parameter $x_i$ is drawn i.i.d. from  $\mathcal{D}$~\citep{l2o,amos_tutorial}. 
This viewpoint naturally casts machine learning problems (\eg, learning initializations or update rules) for the underlying parametric problem.
Learning to optimize has seen success over many domains: \eg, sparse coding~\citep{lista,alista}, meta learning~\citep{maml}, optimal power flow~\citep{fioretto2020predicting}, and---most relevant to this work---convex optimization~\citep{learn_algo_steps}. 

In this paper, we address two limitations of the current literature on learning to optimize for convex optimization. 
\emph{First, it remains open how to obtain high-quality solutions given a finite budget of iterations and a low amount of training data.}
In the literature, various techniques have been proposed to achieve this---\eg, learning metrics for operator-splitting methods~\citep{metric_learning}, learning warm starts~\citep{l2ws}, learning algorithm updates using reinforcement learning~\citep{rlqp}, and learning surrogate problems~\citep{kolouri_learn_cons}.
Yet, these methods are typically not data-efficient in that they require thousands of data points. %
One strategy that has shown to be highly data-efficient is to \emph{only} learn the algorithm hyperparameter sequence, such as the step size sequence in gradient descent~\citep{learn_algo_steps}.
Since momentum is a popular strategy to accelerate first-order methods, it is natural to consider extending it to the learning to optimize framework.

\emph{Second, learned optimizers typically lack worst-case guarantees within a given number of iterations.}
Several works establish generalization guarantees, aiming to ensure that learned optimizers perform well on unseen instances drawn from the training distribution~\citep{balcan_gen_guarantees,sambharya2024data,Sucker2024LearningtoOptimizeWP}.
However, these approaches typically rely on the i.i.d. assumption, which in many real-world scenarios is not realistic or hard to verify.
Other works establish asymptotic convergence guarantees~\citep{safeguard_convex, banert2021accelerated}, which certify that the iterates of the learned optimizer will converge to an optimal solution in the limit.
Yet, such guarantees do not generally give worst-case guarantees given a finite number of iterations.
Given a limited iteration budget from time constraints, certifying such worst-case performance becomes critical.

\myparagraph{Contributions}
In this paper, we propose a framework to solve parametric convex optimization problems quickly while maintaining robustness. 
As our main contributions, (i) we present a framework to learn the hyperparameters of {\bf acceleration} algorithms for a variety of first-order methods, and (ii) we adapt our training problem so that {\bf robustness} with respect to the problem parameter (\ie, numerical worst-case guarantees within a pre-defined number of iterations for all parameters within a given set) can be achieved.
Our key contributions are as follows: 
\begin{itemize}[left=5pt]
  \item {\bf Learning acceleration hyperparameters framework.} %
  We present a machine-learning framework to learn the hyperparameters of momentum-based first-order methods within a provided budget of iterations. 
  In (proximal) gradient descent, we learn the sequence of step sizes and momentum values.
  In the alternating direction method of multipliers (ADMM) and two ADMM-based solvers---the Operator Splitting Quadratic Program (OSQP)~\citep{osqp} and the Splitting Conic Solver (SCS)~\citep{scs_quadratic}---we learn the sequence of relaxation and momentum values, along with a few time-invariant hyperparameters.
  \item {\bf Robustness.} %
  Provided the hyperparameters of these acceleration schemes, we show how to provide numerical worst-case guarantees over a large function class (\eg, minimizing an objective that is convex and $L$-smooth) within a given finite number of iterations.
  This is done by the performance estimation problem (PEP)~\citep{pep} framework, which provides worst-case guarantee by solving a semidefinite program (SDP) where the learned hyperparameters appear as problem data.
  Notably, this gives a worst-case guarantee over all parameters $x$ within a given set $\mathcal{X}$, provided a certain property on the corresponding function class.
  We then show how to train our learned optimizer to achieve a desired level of this worst-case guarantee.
  In order to use gradient-based methods to train our method, we show how to differentiate the objective value of the PEP problem with respect to these hyperparameters.
  \item {\bf Numerical experiments.} We showcase the effectiveness of our approach on a wide variety of numerical experiments from control, signal processing, and statistics.
  We show that acceleration by momentum dramatically improves performance, and that our robustness formulation provides strong worst-case guarantees.
  In most of our examples, our robustness guarantees hold for all possible parameters: \ie, $\mathcal{X} = \reals^d$.
  In some examples, we illustrate the importance of robustness by applying different approaches to out-of-distribution problem instances.
  Our approach is highly data-efficient in that we only use $10$ training instances for each numerical example, while each decision variable is of dimension at least $500$ in all examples.
\end{itemize}

\myparagraph{Layout of paper}
In \Sec~\ref{sec:related_work}, we review related work.
In \Sec~\ref{sec:framework}, we present our framework to learn the acceleration hyperparameters for a variety of first-order methods.
In \Sec~\ref{sec:robustness}, we show how the learned hyperparameters can be used to derive worst-case guarantees using PEP and then how to augment the training problem to encourage this robustness.
In \Sec~\ref{sec:numerical_experiments} we illustrate the effectiveness of our approach with numerical examples, and in \Sec~\ref{sec:conclusion} we conclude.

\myparagraph{Notation}
We denote the set of vectors of length $n$ consisting of real values, nonnegative values, and positive values with $\reals^n$, $\reals^n_+$, and $\reals^n_{++}$ respectively.
We denote the set of positive definite matrices of size $n \times n$ with $\symm^n_{++}$.
For a convex function $g : \reals^n \rightarrow \reals \cup \{+\infty\}$, we denote the proximal operator as $\prox_{g}(v) = \argmin_x g(x) + (1/2)\|x - v\|_2^2$.
For $0 \leq \mu < L \leq \infty$, we denote the set of all $\mu$-strongly convex and $L$-smooth functions as $\mathcal{F}_{\mu,L}$,
and such quadratic functions as $\mathcal{Q}_{\mu,L} = \{z \mapsto (1/2) z^T Q z + c^T z + d \mid Q=Q^T, \mu I \preceq Q \preceq L I\}$.
We denote the set of all $M$-Lipschitz operators as $\mathcal{T}_M$.
For a vector $v \in \reals^n$, we denote the element-wise positive part and sign function as $v_+$ and ${\bf sign}(v)$ respectively.
For a vector $v \in \reals^n$, we define $\textbf{diag}(v)$ as the diagonal matrix with $v$ on the diagonal.
We define the norm of a vector $v \in \reals^n$ with respect to a matrix $R \in \symm^n_{++}$ to be $\|v\|_R = \sqrt{v^T R v}$.
For a convex cone $\mathcal{K}$, we denote its dual cone with $\mathcal{K}^*=\{w \mid w^T z \geq 0, z \in \mathcal{K}\}$.

\section{Related work}\label{sec:related_work} 
\myparagraph{Learning for convex optimization}
A variety of approaches have been developed for learning for convex optimization, including: learning algorithm steps with reinforcement learning~\citep{rlqp}, learning metrics for operator-splitting algorithms~\citep{metric_learning}, learning acceleration steps with recurrent neural network models~\citep{neural_fp_accel_amos}, and learning algorithm updates that are close to known convergent algorithms~\citep{banert2021accelerated}.
A common issue with these methods is that they typically require thousands of training instances, and do not come with worst-case guarantees within a finite number of iterations.
By \emph{only} learning a low number of hyperparameters in each iteration, we find that our method is highly data-efficient and also is amenable to a worst-case analysis based on PEP.

Another line of work learns certain maps from the problem parameter to the initial point~\citep{l2ws_l4dc, l2ws}.
While these methods inherit the same worst-case guarantees as their non-learned counterparts, the maps are often between high-dimensional objects---from parameter in $\reals^d$ to initial point in $\reals^n$.
Such dimensionality may incur poor generalization and require a large amount of training instances~\citep{l2ws}.
Additionally, these methods leave the algorithm untouched; in contrast, learning the algorithm's internal hyperparameters (\eg, step sizes or momentum terms) offers flexibility and control throughout the optimization process.

Our work builds off the work of~\citep{learn_algo_steps}, which only learns a few algorithm hyperparameters in each iteration and is shown to be highly data-efficient.
The main differences of this work are that (i) we also learn hyperparameters that enable momentum-based acceleration and (ii) we augment our training procedure with a regularization term that allows us to obtain worst-case guarantees for all parameters within a set.
Additionally, we enforce some hyperparameters to remain time-invariant in ADMM; this allows us to use PEP to obtain worst-case guarantees while keeping the algorithm to be computationally tractable.

\myparagraph{Learning beyond convex optimization}
The idea of learning to optimize has been explored beyond convex optimization---\eg, to solve inverse and non-convex problems. We defer the readers to the excellent surveys~\citet{l2o,amos_tutorial} for a comprehensive overview on this active line of research. 
We note that this work was in part inspired by certain setting-specific approaches that only learn a few hyperparameters in each iteration, such as for sparse coding~\citep{ablin2019learning, alista} and robust PCA~\citep{cai2021learned}.

\myparagraph{Generalization guarantees for learned optimizers}
Generalization bounds guarantee that a learned optimizer will perform well on unseen problem instances with high probability~\citep{balcan_gen_guarantees,sambharya2024data,Sucker2024LearningtoOptimizeWP}. 
These results are often derived using tools from statistical learning theory, and in some cases the training objective is the generalization bound itself~\citep{sambharya2024data,Sucker2024LearningtoOptimizeWP}.
These results, however, are probabilistic and thus the performance of a learned optimizer can be arbitrarily poor for a particular instance.
They also rest on the i.i.d. assumption, which is often unrealistic (\eg, in control applications where problems are solved sequentially; see~\citet{borrelli_mpc_book}) or difficult to verify in practice. %

\myparagraph{Convergence guarantees for learned optimizers}
Convergence guarantees certify that the iterates of the learned optimizer will eventually reach an optimal solution~\citep{amos_tutorial}.
Such guarantees can be enforced through safeguarding methods~\citep{safeguard_convex}, by constraining iterates to remain within provably convergent regions~\citep{banert2021accelerated,martin2024learning}, or by greedy regularization schemes~\citep{fahy2024greedy}.
Yet, these asymptotic guarantees do not directly translate to guarantees in a finite number of iterations.

Moreover, these approaches typically aim to ensure that every individual step makes progress in terms of minimizing the objective.
However, recent theoretical results (\eg, \citet{grimmer2024provably, altschuler2023acceleration_str_cvx}) demonstrated that the overall progress can be improved even if not all individual steps are making progress. Notably, such works use step sizes that are occasionally very large---a pattern similarly observed for learned step sizes~\citep{learn_algo_steps}. While such large step sizes are individually suboptimal, they improve performance when combined with other step sizes. Our framework does not limit each step to make progress and thus can capture such phenomenon, enabling further acceleration.

\myparagraph{Worst-case guarantees for parametric optimization}
In this paper, we seek a strong form of robustness guarantee---certification of worst-case performance for all parameters within a given set after a given number of iterations of a first-order method.
This certification problem has been studied via an exact mixed-integer programming formulation~\citep{ranjan2024exact} and an SDP relaxation~\citep{perfverifyqp}.
However, their sizes scale with the dimension of the problem, and both are specifically tailored to quadratic programs.
We overcome this by using the PEP framework, which is dimension-independent and applies to more general function classes. 

\myparagraph{Performance estimation}
The PEP framework~\citep{pep} is a computer-assisted analysis that computes the worst-case guarantee for a first-order method within a provided number of iterations by solving an SDP. 
This framework has numerous applications, such as obtaining improved convergence rates and designing new algorithms; see \citet{accel_survey, taylorhabilitation} for an overview.
Notably, PEP provides a \emph{tight} worst-case guarantee for various function classes: \eg, functions that are (strongly) convex and smooth \citep{pep2}, quadratic \citep{bousselmi2024interpolation}, or nonexpansive \citep{ryu_ospep}. 
In this work, we regularize our training problem with the optimal value of the SDP that is derived from PEP, thereby certifying worst-case guarantees for parameters within a given set. %

\section{Learning acceleration hyperparameters framework}\label{sec:framework}

In this section, we introduce our approach to learn the acceleration hyperparameters of several first-order methods.
Throughout, we denote the learned hyperparameters (also called the weights) by $\theta$.
In \Subsec~\ref{subsec:running}, we show how to run the learned optimizer provided the weights $\theta$, and then in \Subsec~\ref{subsec:training} we show how to train the weights $\theta$ in order to minimize a performance loss.
Later, in \Sec~\ref{sec:robustness}, we will show how to certify worst-case guarantees based on the hyperparameters $\theta$ and then how to augment the training problem in order to reach a desired worst-case guarantee.

\subsection{Accelerated algorithms and hyperparameters}\label{subsec:running}
In this subsection, we show how to run our learned optimizer provided the weights $\theta$.
The weights $\theta$ are broken into two groups: \emph{time-varying} hyperparameters denoted by $\theta^k \in \reals^2$ for the $k$-th iteration, and \emph{time-invariant} hyperparameters denoted by $\theta^{\rm inv} \in \reals^{n^{\rm inv}}$.
Thus, the full set of weights is given by $\theta = (\theta^0,\dots,\theta^{K-1}, \theta^{\rm inv})$.
We consider the following form for accelerated algorithms: 
\begin{equation}\label{eq:accelerated_iterates}
        y^{k+1}_\theta(x) = T_{\phi^k}(z^k_\theta(x), x), \quad
        z^{k+1}_\theta(x) = y^{k+1}_\theta(x) + \beta^k\left(y^{k+1}_\theta(x) - y^{k}_\theta(x)\right).
\end{equation}
Here, $T_{\phi^k}$ denotes the base fixed-point operator from~\eqref{prob:fp_algo}, parameterized by weights $\phi^k = (\alpha^k,\theta^{\rm inv})$.
We apply our framework to gradient descent, proximal gradient descent, ADMM, and two ADMM-based solvers: OSQP and SCS.
In all cases, the time-varying weights at the $k$-th step are $\theta^k = (\alpha^k,\beta^k)$.
For each of these, we provide the learned accelerated variant in Table~\ref{table:fp_algorithms}. 
The subscript $\theta$ is used to emphasize that the iterates depend on the weights.

\begin{table}[!h]
    \centering
    \caption{Several popular first-order methods and their hyperparameters in our learned acceleration framework.
    The iterates we write are the \emph{accelerated} versions of the base algorithms and fall under the general form of~\eqref{eq:accelerated_iterates}.
    See Appendix \Subsec~\ref{sec:fom_details} for further details.
    }
    \label{table:fp_algorithms}
    \adjustbox{max width=\textwidth}{
    \begin{threeparttable}
    \begin{tabular}{@{}lllll@{}}
      \toprule[\heavyrulewidth]
      \makecell[l]{Base \\Algorithm} & Problem &  Iterates & \makecell[l]{Time-varying \\Hyperparameters}  & \makecell[l]{Time-invariant \\Hyperparameters}\\
      \midrule[\heavyrulewidth]
      \makecell[l]{Gradient \\ descent } & $\begin{array}{@{}ll}\min &f(z,x)\end{array}$ & $\begin{aligned}y^{k+1} &= z^k - \alpha^k \nabla f(z^k,x) \\ z^{k+1} &= y^{k+1} + \beta^k (y^{k+1} - y^k) \end{aligned}$  & $\theta^k = (\alpha^k, \beta^k)$ & $\theta^{\rm inv} = ()$\\
      \midrule
      \makecell[l]{Proximal \\ gradient \\ descent }
      & $\begin{array}{@{}ll}\min & f(z,x)+ g(z,x)\end{array}$ & $\begin{aligned}y^{k+1} &= \prox_{\alpha^k g}(z^k - \alpha^k \nabla f(z^k,x)) \\ z^{k+1} &= y^{k+1} + \beta^k (y^{k+1} - y^k) \end{aligned}$ & $\theta^k = (\alpha^k, \beta^k)$ & $\theta^{\rm inv} = ()$\\
      \midrule
      \makecell[l]{ADMM\\ \citep{dr_splitting, Boyd_admm}}
      & $\begin{array}{@{}ll}\min & f(w,x)+ g(w,x)\end{array}$ & $\begin{aligned}w^{k+1} &= \prox_{\eta f}(z^k) \\ \tilde{w}^{k+1} &= \prox_{\eta g}(2 w^{k+1}- z^k) \\ y^{k+1} &= z^k + \alpha^k (\tilde{w}^{k+1} - w^{k+1}) \\ z^{k+1} &= y^{k+1} + \beta^k (y^{k+1} - y^k)\end{aligned}$ & $\theta^k = (\alpha^k, \beta^k)$ & $\theta^{\rm inv} = (\eta)$\\
      \midrule
      \makecell[l]{OSQP \\ \citep{osqp}}\textcolor{black}{} & $\begin{array}{@{}ll}
      \min & (1/2) w^T Pw + c^T w\\
      \mbox{s.t.}& l \leq Aw \leq u \quad \text{dual\ } (s) \\
      &\\&\\
      \text{with}&x = (P, A, c, l, u)\\
      \end{array}$
      &
      $\begin{aligned}
        & (w^k, \xi^k)= z^k\\
        & v^{k+1} = \Pi_{[l,u]}(\xi^k)\\
        & \text{solve } Q \tilde{w}^{k+1} = \sigma w^k - c + \diag (\rho) (A^T (2 v^{k+1} - \xi^k))\\
        & \tilde{\xi}^{k+1} = \diag (\rho) (A w^{k+1} + \xi^{k} - 2 v^{k+1}) + \xi^k - v^{k+1} \\
        & y^{k+1} = (w^k + \alpha^k (\tilde{w}^{k+1} - w^k), \xi^k + \alpha^k (\tilde{\xi}^{k+1} - \xi^k))\\
        & z^{k+1} = z^k + \beta^k (y^{k+1} - y^k)
        \\ \\
        & \text{with}\quad Q = P + \sigma I + \diag (\rho)  A^T A\\
        &\rho = (\rho_{\rm eq} \mathbf{1}_{m_{\rm eq}}, \rho_{\rm ineq} \mathbf{1}_{m_{\rm ineq}})\\
      \end{aligned}$ & $\theta^k = (\alpha^k, \beta^k)$ & $\theta^{\rm inv} = (\sigma, \rho_{\rm eq}, \rho_{\rm ineq})$\\
      \midrule
      \makecell[l]{SCS \\ \citep{scs_quadratic}}\textcolor{black}{} & $\begin{array}{@{}ll}
    \min & (1/2) w^T Pw + c^T w\\
    \mbox{s.t.}& Aw + s = b \quad \text{dual\ } (v) \\
    & s \in \mathcal{K} \\
    &\\
      \text{with}&x = (P, A, c, b)
    \end{array}$ 
    & $\begin{aligned}
      & \text{solve } (R + M) \tilde{u}^{k+1} = R (z^k - q) \\
      & u^{k+1} = \Pi_{ \reals^p \times \mathcal{K^*}}(2 \tilde{u}^{k+1} - z^k)\\
      & y^{k+1} = z^k + \alpha^k (u^{k+1} - \tilde{u}^{k+1})\\
      & z^{k+1} = y^{k+1} + \beta^k (y^{k+1} - y^k)\\ \\
      & \text{with}\quad 
      R = \diag(r_w \mathbf{1}_q, r_{y_{\rm z}} \mathbf{1}_{m_{\rm z}}, r_{y_{\rm nz}}  \mathbf{1}_{m_{\rm nz}}) \\
      &     M = \begin{bmatrix}
        I_p + P & A^T  \\
        -A & I_m \\
      \end{bmatrix}, q = (c,b)\\
    \end{aligned}$ & $\theta^k=(\alpha^k, \beta^k)$ & $\theta^{\rm inv} = (r_w, r_{y_{\rm z}}, r_{y_{\rm nz}})$\\
    \bottomrule[\heavyrulewidth]
    \end{tabular}
    \end{threeparttable}
    }
  \end{table}

\myparagraph{Connection to Nesterov acceleration}
In Table~\ref{table:fp_algorithms}, the accelerated gradient descent and accelerated proximal gradient descent algorithms are generalizations of Nesterov's acceleration respectively for smooth convex optimization and composite optimization.
By setting $\alpha^k$ and $\beta^k$ appropriately, we can recover Nesterov's method. 
  
\myparagraph{Connection to Anderson acceleration}
Since Nesterov's acceleration method is limited to (proximal) gradient descent, developing acceleration schemes for more general fixed-point problems remains an open research direction.
One popular heuristic that can work well in practice is Anderson acceleration~\citep{anderson1965iterative}, which combines multiple past iterates.
To be specific, Anderson acceleration computes the next iterate as a weighted combination of the last $H$ fixed-point updates, where the weights are chosen to sum to one.
The weights are typically computed by solving an equality-constrained least-squares problem at each iteration~\citep[Problem~1.1]{anderson_acceleration}. 
Our approach can be seen as a learned variant of Anderson acceleration, where we fix $H=2$ and replace the per-iteration optimization of weights with a learned schedule that is shared across all problem instances.
Leveraging the parametric setting allows us to (i) learn all hyperparameters jointly, instead of tuning weights one iteration at a time, and (ii) certify worst-case performance via PEP as we will show in \Sec~\ref{sec:robustness}.

\myparagraph{The purpose of time-invariant hyperparameters in ADMM}
In our framework, we enforce certain hyperparameters to be time-invariant for ADMM-based algorithms. The main reason is to formulate these algorithms as accelerated fixed-point iteration with respect to the \emph{same} nonexpansive operator throughout each step. We formally present this in the following proposition.

\begin{proposition}\label{prop:admm_nonexp}
  Let $x$ be the problem parameter and let $\{y^k(x),z^k(x)\}_{k=0,1\dots}$ denote the iterates of an ADMM-based algorithm from Table~\ref{table:fp_algorithms} with weights $\theta$.
  Let the time-invariant parameters $\theta^{\rm inv}$ be positive.
  Then there exist a matrix $R(\theta^{\rm inv}) \in \symm_{++}^n$ and an operator $S_{\theta^{\rm inv}} :\reals^n \times \reals^d \rightarrow \reals^n$ that is nonexpansive in its first argument with respect to $R(\theta^{\rm inv})$ such that
  \begin{align*}
  y^{k+1}(x) &= (\alpha^k / 2) S_{\theta^{\rm inv}}(z^k(x),x) + (1 - (\alpha^k / 2)) z^k(x), \\
  z^{k+1}(x) &= y^{k+1}(x) + \beta^k (y^{k+1}(x) - y^k(x)).
\end{align*}
\end{proposition}
See Appendix~\Subsec~\ref{proof:admm_nonexp} for the proof.
Notably, this comes with two significant benefits. 
First, as we will elaborate in \Sec~\ref{sec:robustness}, the fact that $S_{\theta^{\rm inv}}$ is the same across iterations makes the method well-suited for a worst-case analysis with PEP. 

Another benefit lies on computational tractability.
For solving convex conic programs,
running OSQP or SCS without any learned hyperparameters requires factoring a matrix once and then solving a linear system in each iteration with back-substitution by reusing the same factorization but with a different right-hand-side vector~\citep{osqp,scs_quadratic}. 
Factoring a dense matrix (\eg, with an LU-factorization) has complexity $\mathcal{O}(n^3)$, while the back-substitution step has complexity $\mathcal{O}(n^2)$~\citep[\Sec~11]{vmls}.
In particular, if the hyperparameters that enter the matrix in the linear system change across iterations, then at each iteration we must factor a new matrix.
On the other hand, fixing the hyperparameters that enter the system matrix across all iterations enables reuse of a single matrix factorization, thereby greatly reducing the computational burden.

\subsection{Training acceleration hyperparameters}\label{subsec:training}
In this subsection, we explain how to train the weights $\theta$ to improve performance over the parametric family of problems for a budget of $K$ iterations. 
We first define the loss functions that we use and then formulate the training problem.

\myparagraph{Loss functions}
The loss function $\ell(z^K_\theta(x), x)$ (with respect to the $K$-th iterate $z^K_\theta(x)$ for parameter $x$)
depends on the structure of the underlying optimization problem.
In particular, we distinguish between two different cases:
\begin{itemize}[left=5pt]
  \item {\bf Unconstrained problems}. We use the suboptimality $f(z^K_\theta(x),x) + g(z^K_\theta(x), x) - f(z^\star(x), x) - g(z^\star(x), x)$, a standard measure of performance.
  \item {\bf Constrained problems}. 
  We use the sum of the primal and dual residuals, which upper-bounds the standard max-residual metric used in constrained convex solvers~\citep{osqp,scs_quadratic}.
\end{itemize}
Technically speaking, our method requires optimal solutions for the training instances in the case of unconstrained optimization.
While this requirement might seem limiting, for most of the problems of our interest the main computational bottleneck is to train the learned optimizer rather than to solve all of the training instances.
In particular, for all of our examples we only need to solve ten training instances.

\myparagraph{The learning acceleration hyperparameters training problem}
We assume access to a training dataset of parameters and corresponding optimal solutions $\{(x_i, z^\star(x_i))\}_{i=1}^N$.
We formulate our problem of learning acceleration hyperparameters as
\begin{equation}\label{prob:simplified_l2o}
  \begin{array}{ll}
\mbox{minimize} & (1 / N) \sum_{i=1}^N \ell(z^K_\theta(x_i),x_i)\\
  \mbox{subject to} &y_\theta^{k+1}(x_i) = T_{\phi^k}(z^k_\theta(x_i)) ,\hspace{2mm} k\leq  K-1, \hspace{2mm}  i \leq N-1\\
      &z_\theta^{k+1}(x_i) = y_\theta^{k+1}(x_i) + \beta^k (y_\theta^{k+1}(x_i) - y_\theta^k(x_i)),  \hspace{1.9mm} k \leq K-1, \hspace{1.9mm} i \leq N-1\\\
  &z^0_\theta(x_i) = 0,  y^0_\theta(x_i) = 0, \hspace{2mm}  i \leq N-1,
\end{array}
\end{equation}
where $\theta = (\alpha^0, \beta^0, 
\dots, \alpha^{K-1}, \beta^{K-1}, \theta^{\rm inv}) \in \reals^{2K + n^{\rm inv}}$ is the decision variable.

\myparagraph{Using gradient-based methods to train}
In order to train, we unroll~\citep{algo_unrolling} the algorithm steps---\ie, we differentiate through them.
To compute such gradients, we rely on autodifferentiation techniques~\citep{autodiff}.\footnote{The training problem~\eqref{prob:simplified_l2o} is in general not differentiable at every point; indeed, in many cases the fixed-point operator involves a non-differentiable step (\eg, the projection step onto the nonnegative orthant in projected gradient descent).
At such points, the autodifferentiation techniques use subgradients to estimate directional derivatives~\citep{autodiff}.}
In our framework, it is necessary that some of the weights are positive (see Proposition~\ref{prop:admm_nonexp}). 
For this, we can reparametrize as $\theta_i^k = \exp(\nu_i^k)$ and optimize for $\nu_i^k \in \reals$.

\section{Worst-case certification and PEP-regularized training}\label{sec:robustness}
The learning to optimize problem is designed to minimize the loss over the distribution of problem parameters $\mathcal{D}$, but does not provide the worst-case guarantees we seek.
In this section, we first show in \Subsec~\ref{subsec:eval_robustness} how the choice of hyperparameters can be used to provide worst-case guarantees for all possible parameters $x$ within a set $\mathcal{X}$.
Our method is grounded in the PEP framework~\citep{pep,pep2}, an SDP-based technique that provides numerical worst-case guarantees within a given number of iterations.
In \Subsec~\ref{subsec:train_robustness} we augment the learning to optimize training problem~\eqref{prob:simplified_l2o} to train our hyperparameters so that our learned optimizer can both perform well over the distribution of parameters while also obtaining worst-case guarantees for any parameter $x \in \mathcal{X}$.
This culminates with the formulation of the \emph{PEP-regularized training problem}~\eqref{prob:pep_regularized}.
We show how to differentiate through the solution of this SDP with respect to the hyperparameters in order to use gradient-based training.
Finally, in \Subsec~\ref{subsec:pros_cons}, we discuss advantages and limitations of using PEP in our framework.

\subsection{Semidefinite programming for worst-case certification}\label{subsec:eval_robustness}

In this subsection, we show how to obtain worst-case guarantees for parameters $x \in \mathcal{X}$ provided the learned hyperparameters $\theta$.
These guarantees take the form 
\begin{equation}\label{eq:worst_case}
  r(z^K_\theta(x), x) \leq \gamma(\theta) \|z^0(x) - z^\star(x)\|^2 \quad \forall x \in \mathcal{X},
\end{equation}
where $r : \reals^n \times \reals^d \rightarrow \reals_+$ is a performance metric (\eg, distance to optimality) and $z^\star(x)$ is any minimizer.

\subsubsection{Worst-case certification of accelerated proximal gradient descent}

We now present our method for obtaining worst-case guarantees for accelerated (proximal) gradient descent.
Our analysis is based on PEP, which provides numerical worst-case guarantees for general \emph{function classes} such as $\mathcal{F}_{\mu, L}$---\ie, $\mu$-strongly convex and $L$-smooth functions. We formalize this for our setting as follows.

\begin{definition}[$\mathcal{G}$-parameterized pairs]
  Let $f : \reals^n \times \reals^d \rightarrow \reals \cup \{+\infty\}$ be a function, $\mathcal{X} \subseteq \reals^d$ be a set and $\mathcal{G}$ be a function class. We call that $(f,\mathcal{X})$ is $\mathcal{G}$-parameterized if $f(\cdot,x) \in \mathcal{G} \quad \forall x \in \mathcal{X}$. 
\end{definition}

As an illustration, consider the following example.

\begin{example}\label{ex:least_squares}
  Let $f(z,x) = (1/2)\|Az - x\|_2^2$ and $\mathcal{X} = \reals^d$.
  The pair $(f,\mathcal{X})$ is $\mathcal{Q}_{\mu,L}$-parameterized, where $\mu$ and $L$ are the smallest and largest eigenvalues of $A^T A$. 
\end{example}

For $\mathcal{G}$-parameterized pairs, the worst-case guarantees for the entire function class $\mathcal{G}$ immediately imply worst-case guarantees for all parameters $x \in \mathcal{X}$. 
Such structural property allows us to bypass analyzing the parametric dependence directly, and instead certify worst-case guarantees over the entire function class (which includes all possible problem instances in the set).
To be specific, we consider the following function classes and performance metrics.

\begin{assumption}\label{assumption:proxgd} $\mathcal{G}$ is a function class and $r$ is a performance metric, where: %
\begin{itemize}[leftmargin=5em]
    \item[Case 1:] $\mathcal{G} = \mathcal{F}_{0, L}$,  $r(z, x) = f(z, x) + g(z, x) - f(z^\star(x), x) - g(z^\star(x), x)$,
    \item[Case 2:] $\mathcal{G} = \mathcal{F}_{\mu, L}$, $r(z, x) = \|z - z^\star(x)\|^2$,
    \item[Case 3:] $\mathcal{G} = \mathcal{Q}_{\mu, L}$, $r(z, x) = \|z - z^\star(x)\|^2$.
\end{itemize}
\end{assumption}

\myparagraph{SDP for accelerated proximal gradient descent} Let $(f, \mathcal{X})$ be $\mathcal{G}$-parameterized. Then an auxiliary optimization problem for \eqref{eq:worst_case} can be formulated as follows:
\begin{equation}\label{prob:pep_pgd}
  \begin{array}{lll}
\mbox{maximize} & \mbox{(performance metric)} & r(z^K)\\
  \mbox{subject to}  & \mbox{(initial point)} & z^0 = y^0, \|z^0-z^\star\|_2^2 \leq 1 \\
  & \mbox{(optimality)} & \nabla f(z^\star) + \partial g (z^\star) = 0 \\

  & \mbox{(algorithm update)} & y^{k+1} = \prox_{\alpha^k g}(z^k - \alpha^k \nabla f(z^k)), \hspace{2mm} k \leq K-1 \\
      & & z^{k+1} = y^{k+1} + \beta^k (y^{k+1} - y^k), \hspace{2mm} k \leq K-1\\
  & \mbox{(function class)} & f \in \mathcal{G}, g \in \mathcal{F}_{0,\infty}.
\end{array}
\end{equation}
Letting $\bar{\gamma}(\theta)$ be the optimal value of \eqref{prob:pep_pgd},\footnote{Here, we assume that $\theta$ is given, and the decision variables to the maximization problem are the functions $f$ and $g$, and the iterates $z^0,\dots,z^K,z^\star,y^0,\dots,y^K$. From $(f, \mathcal{X})$ being $\mathcal{G}$-parametrized, we remove the notational dependency on the parameter $x$ from the functions $f$, $g$, and $r$ for simplicity.} we have $r(z^K) \leq \bar{\gamma}(\theta)\|z^0 - z^\star\|^2$ (by rescaling, under Assumption \ref{assumption:proxgd}), which is the worst-case guarantee that we aim for.

However, \eqref{prob:pep_pgd} cannot be directly solved because the function class constraints are infinite-dimensional. 
The key idea of PEP is to replace that condition on the function classes with a finite number of valid inequalities, for all pairs of the iterates evaluated throughout the algorithm update \citep{pep, taylor2017exactworst, pep2}. 
This yields an SDP formulation, where a positive semidefinite matrix encodes the Gram matrix of the algorithm iterates. 
For simplicity, here we consider the case $\mathcal{G} = \mathcal{F}_{\mu, L}$; see Appendix Subsection \ref{appendix:sdp_pgd} for the similar case when $\mathcal{G} = \mathcal{Q}_{\mu, L}$.
The SDP is formulated as
\begin{equation}\label{prob:sdp_pgd}
  \begin{array}{lll}
\mbox{maximize} & \mbox{(performance metric)} & \Tr(G U) + \sum_i (v^i f^i + w^ig^i) \\
  \mbox{subject to} & \mbox{(initial point)} & \Tr(G A^0) \leq 1 \\
  & \mbox{(optimality)} & \Tr(GA^\star) = 0 \\
  & \mbox{(algorithm update} & f^j \geq f^i + \Tr(GB^{ij}) \quad \forall i, j \\
  & \mbox{+ function class)} & g^j \geq g^i + \Tr(G C^{ij}) \quad \forall i, j \\
  & \mbox{(Gram matrix)} & G \succeq 0,
  \end{array}
\end{equation}
for certain choices of the coefficients (see Appendix Subsection \ref{appendix:sdp_pgd} for details).
Here, the decision variables are $f^0, \dots, f^K, f^\star, g^0, \dots, g^K, g^\star \in \reals$ (corresponding to functional values of $f$ and $g$) and $G$ (Gram matrix of the iterates). 
For the function classes of our interest, several works have characterized interpolation inequalities that are tight \citep{taylor2017exactworst, pep2, bousselmi2024interpolation}, which we use here as the valid inequalities; in other words, the SDP relaxation \eqref{prob:sdp_pgd} attains the same optimal value as the original problem \eqref{prob:pep_pgd}. In particular, the optimal value $\gamma(\theta)$ of the SDP \eqref{prob:sdp_pgd} satisfies the following.

\begin{theorem}
  Suppose that Assumption \ref{assumption:proxgd} holds for function class $\mathcal{G}$ and performance metric $r$. 
  Also, assume that $(f, \mathcal{X})$ is $\mathcal{G}$-parametrized, $(g, \mathcal{X})$ is $\mathcal{F}_{0, \infty}$-parametrized, and an optimal solution $z^\star(x)$ exists for all $x \in \mathcal{X}$.
  Then for the optimal value $\gamma(\theta)$ of \eqref{prob:sdp_pgd},
  \begin{equation*}
    r(z^K_\theta(x),x) \leq \gamma(\theta) \|z^0(x) - z^\star(x)\|_2^2 \quad \forall x \in \mathcal{X}.
  \end{equation*}
\end{theorem}
\begin{proof}
    This follows because the constraints in \eqref{prob:pep_pgd} imply those in \eqref{prob:sdp_pgd} from interpolation inequalities \citep{taylor2017exactworst, pep2,  bousselmi2024interpolation}; for details, see Appendix Subsection \ref{appendix:sdp_pgd}.
    The worst-case bounds on the function class then imply worst-case bounds over $\mathcal{X}$ since $(f,\mathcal{X})$ is $\mathcal{G}$-parametrized and $(g,\mathcal{X})$ is $\mathcal{F}_{0,\infty}$-parametrized.
\end{proof}
Note that this result also applies to accelerated gradient descent (\ie, when $g \equiv 0$).

\subsubsection{Worst-case certification of accelerated ADMM}
The PEP formulation for ADMM is fairly similar to that for accelerated proximal gradient descent.
Our approach starts with identifying ADMM as a fixed-point iteration; recall from Proposition~\ref{prop:admm_nonexp}, that the iterations of accelerated ADMM (and ADMM-based solvers) from Table~\ref{table:fp_algorithms} can be represented in terms of fixed-point iterations that are nonexpansive with respect to a positive definite matrix.

\begin{definition}[$R$-nonexpansive pairs]
  Let $S : \reals^{n} \times \reals^d \rightarrow \reals^n$ be an operator, $R \in \symm^n_{++}$, and $\mathcal{X} \subseteq \reals^d$ be a set of parameters.
  We call that $(S, \mathcal{X})$ is $R$-nonexpansive if for all $x \in \mathcal{X}$, the operator $S(\cdot,x)$ is nonexpansive with respect to $R$.
\end{definition}

\myparagraph{SDP for accelerated ADMM}

We obtain the worst-case guarantee in terms of the fixed-point residual as this is a standard metric for ADMM~\citep[\Sec~7.3]{mon_primer}. 
Provided the hyperparameters $\theta$ where $\theta^{\rm inv}$ is positive, Proposition~\ref{prop:admm_nonexp} guarantees the existence of a matrix $R(\theta^{\rm inv}) \in \symm^n_{++}$ such that the algorithm can be written as fixed-point iterations that are nonexpansive with respect to $R(\theta^{\rm inv})$, combined with averaging steps (with $\alpha$ values) and momentum steps (with $\beta$ values).

We now turn to writing the certification problem as we did in the previous subsection.
Note that we suppress the dependence on $x$ and $\theta$ for simplicity in the formulation, and that the underlying norm is $\| \cdot \|_{R(\theta^{\rm inv})}$; in particular, the performance metric is $r(z, x) = \|z - S(z,x)\|_{R(\theta^{\rm inv})}^2$ and the guarantee is with respect to $\|z^0(x) - z^\star(x)\|_{R(\theta^{\rm inv})}^2$.
We formulate the optimization problem as
\begin{equation}\label{prob:pep_admm}
  \begin{array}{lll}
\mbox{maximize} & \mbox{(performance metric)} & \|z^K - S(z^K) \|^2 \\
  \mbox{subject to} & \mbox{(initial point)} &z^0 = y^0, \|z^0-z^\star\|^2 \leq 1 \\
  & \mbox{(optimality)} &z^\star = S(z^\star) \\
  & \mbox{(algorithm update)} & y^{k+1} = (1 - \alpha^k) z^k + \alpha^k S(z^k), \quad k \leq K-1 \\
      & &z^{k+1} = y^{k+1} + \beta^k (y^{k+1} - y^k), \quad k \leq K-1 \\

  & \mbox{(operator class)} & S \in \mathcal{T}_1,
\end{array}
\end{equation}
where the decision variables are the iterates $z^0,\dots,z^K,z^\star,y^0,\dots,y^K$ and the operator $S$.
An SDP relaxation of \eqref{prob:pep_admm} can be written as:
\begin{equation}\label{prob:sdp_admm}
  \begin{array}{lll}
\mbox{maximize} & \mbox{(performance metric)} & \Tr(G U) \\
  \mbox{subject to} & \mbox{(initial point)} & \Tr(G A^0) \leq 1 \\ 
  & \mbox{(optimality)} & \Tr(GA^\star) = 0 \\
  & \mbox{(algorithm update + operator class)} & \Tr(G B^{ij}) \leq 0 \quad \forall i, j \\
  & \mbox{(Gram matrix)} & G \succeq 0,
  \end{array}
\end{equation}
for certain choices of $U, A^0, A^\star, B^{ij}$ (that are different from previous section), where the decision variable is the Gram matrix $G$. 

\begin{theorem}
Let $\theta$ be the hyperparameters for an ADMM-based algorithm such that $\theta^{\rm inv} > 0$.
Let $S_{\theta^{\rm inv}}$ and $R(\theta^{\rm inv})$ be the nonexpansive operator and positive definite matrix from Proposition~\ref{prop:admm_nonexp} respectively. 
Assume that for all $x \in \mathcal{X}$, there exists a fixed-point $z^\star(x) = S_{\theta^{\rm inv}}(z^\star(x),x)$.
Then for the optimal value $\gamma(\theta)$ of \eqref{prob:sdp_admm},
  \begin{equation*}
    \|z^K_\theta(x) - S_{\theta^{\rm inv}}(z^K_\theta(x),x)\|_{R(\theta^{\rm inv})}^2 \leq \gamma(\theta) \|z^0(x) - z^\star(x)\|_{R(\theta^{\rm inv})}^2 \quad \forall x \in \mathcal{X}.
  \end{equation*}
\end{theorem}
\begin{proof}
    After using Proposition~\ref{prop:admm_nonexp}, the derivation is fairly similar to that for accelerated proximal gradient descent.
    The constraints in~\eqref{prob:pep_admm} imply those in~\eqref{prob:sdp_admm} from interpolation inequalities~\citep{taylor2017exactworst, pep2}.
    For details, see Appendix Subsection \ref{appendix:sdp_admm}.
\end{proof}
We remark that the SDP used to compute $\gamma(\theta)$ does not depend on the matrix $R(\theta^{\rm inv})$.
This flexibility is particularly useful in our case, since we learn the time-invariant hyperparameters $\theta^{\rm inv}$ that determine $R(\theta^{\rm inv})$.
Such dependence of the metric on $R(\theta^{\rm inv})$ connects to a broader line of work on variable-metric methods which adapts the geometry of the algorithm to improve performance~\citep{ giselsson_lin_conv, lscomo}, and more recent approaches that learn problem-specific metrics from data~\citep{metric_learning}.

\subsection{Augmenting the training problem for robustness}\label{subsec:train_robustness}
In this subsection, we show how to adjust the learning to optimize training problem~\eqref{prob:simplified_l2o} to obtain the learning to optimize \emph{PEP-regularized} training problem.
We then show how to compute the derivative of the optimal objective value of the PEP.

\myparagraph{The PEP-regularized training problem}
We design a training problem that aims to achieve the target value $\gamma^{\rm target}$ specified by the user for the robustness.
Ideally, we would constrain the weights $\theta$ so that $\gamma(\theta) \leq \gamma^{\rm target}$.
As this constraint is difficult to enforce, we instead take a penalty-based approach and formulate the PEP-regularized training problem as
\begin{equation}\label{prob:pep_regularized}
  \begin{array}{ll}
\mbox{minimize} & (1 / N) \sum_{i=1}^N \ell(z^K_\theta(x_i),x_i) + \lambda ((\gamma(\theta) - \gamma^{\rm target})_+)^2\\
  \mbox{subject to} &y_\theta^{k+1}(x_i) = T_{\phi^k}(z^k_\theta(x_i)) ,\hspace{2mm} k \leq K-1, \hspace{2mm} i \leq N-1\\
      &z_\theta^{k+1}(x_i) = y_\theta^{k+1}(x_i) + \beta^k (y_\theta^{k+1}(x_i) - y_\theta^k(x_i)),  \hspace{1.9mm} k \leq K-1, \hspace{1.9mm} i \leq N-1\\\
  &z^0_\theta(x_i) = 0,  y^0_\theta(x_i) = 0, \hspace{1.5mm} i \leq N-1,
\end{array}
\end{equation}
with decision variable $\theta$.
Here, $\lambda \in \reals_{++}$ controls the trade-off between the average loss over the training instances and the worst-case value.
The penalty term is designed so that only $\gamma(\theta)$ values above $\gamma^{\rm target}$ are penalized.
We square the positive part of this difference to ensure smoothness.

\myparagraph{Computing the derivative through the PEP-regularized training problem}
The augmented training problem~\eqref{prob:pep_regularized} is related to the original training problem~\eqref{prob:simplified_l2o} that does not consider robustness, but now $\gamma(\theta)$ involves solving an SDP.
To enable the use of gradient-based methods to solve problem~\eqref{prob:pep_regularized}, we rely on techniques from differentiable optimization~\citep{diff_opt,diffcp2019}.
In each iteration of gradient descent used to solve problem~\eqref{prob:pep_regularized}, we only need to solve \emph{a single} SDP as opposed to solving an SDP for each training instance. %
After the training is complete and the weights $\theta$ are fixed, we can compute the exact $\gamma(\theta)$ value by solving one final SDP.
This value is designed to be close to $\gamma^{\rm target}$ due to the penalty formulation.

\subsection{Advantages and limitations of using PEP}\label{subsec:pros_cons}
We highlight several advantages of using the PEP framework to obtain our worst-case guarantees. 
First, it is scalable in that the corresponding SDP is independent of the dimensions $n$ and $d$ \citep{pep}.
Second, it allows $\mathcal{X}$ to be a large set, such as $\reals^d$ (see Example~\ref{ex:least_squares}).
Third, PEP is compatible with gradient-based training via differentiable optimization, since it is based on a convex SDP formulation. 
Finally, the PEP approach allows us to bound the worst-case performance given the \emph{entire sequence} of learned hyperparameters. 
This enables learning hyperparameters that may be individually suboptimal for a particular step, but can enable acceleration when combined with other steps.

Despite these strengths, using PEP in our setting comes with a few limitations.
First, while the SDP is independent of the dimensions $n$ and $d$, its size grows with the number of iterations $K$.
Second, while the PEP approach is known to provide tight guarantees over function classes~\citep{pep2}, these guarantees may be conservative in our case. 
In particular, we always initialize from the zero vector and we only consider parameters within a set $\mathcal{X}$; while PEP applies to these settings, it does not take advantage from such specificities.

\section{Numerical experiments}\label{sec:numerical_experiments}
In this section, we show the strength of our robustly learned acceleration methods via many numerical examples.\footnote{\url{https://github.com/rajivsambharya/learn_accel_steps_robust} contains the code to reproduce our results.}
We apply our method to gradient descent in \Subsec~\ref{subsec:num_gd}, proximal gradient descent in \Subsec~\ref{subsec:num_prox_gd}, OSQP in \Subsec~\ref{subsec:num_osqp}, and SCS in \Subsec~\ref{subsec:num_scs}.
We use JAX~\citep{jax} to train our hyperparameters using the Adam~\citep{adam} optimizer.
To solve the SDP in the PEP-regularized training problem~\eqref{prob:pep_regularized}, we use the first-order method SCS~\citep{scs_quadratic}, and to differentiate through this SDP, we use Cvxpylayers~\citep{cvxpylayers2019}.
In all of our examples, we train on only $10$ instances and evaluate on $1000$ unseen test instances.

\myparagraph{Robustness levels}
In each example, we first apply our approach without training for any worst-case guarantees (\ie, we set $\lambda = 0$ in \eqref{prob:pep_regularized}).
Then, we pick robustness values for $\gamma^{\rm target}$, and train our method, setting the regularization coefficient for the robustness penalty to be $\lambda = 10$.
Since our method is penalty-based, the robustness levels approximately reach the target values in all cases.

\myparagraph{Baseline comparisons}
We refer to our method proposed in this work as {\bf LAH Accel} (short for learned algorithm hyperparameters -- acceleration).
For each example, we provide the final $\gamma(\theta)$ values (which approximately match the chosen $\gamma^{\rm target}$ values).
Within each example, for all methods we train for the same number of iterations $K$.
We are primarily interested in the performance after $K$ steps; we also report, for each example, the number of iterations required to meet several accuracy tolerances, giving a fuller picture of convergence speed.
Our plots report the geometric mean of the performance metric over the test instances.\footnote{For a vector $v \in \reals^m_{++}$ the geometric mean is given by $(\Pi_i v_i)^{1 / m}$.
Using the geometric mean is common when analyzing the performance of optimization algorithms~\citep{osqp,learn_algo_steps}.
}
After $K$ steps, we run the vanilla method (see below) for LAH Accel. %
We compare against the following methods:
\begin{itemize}[left=5pt]
  \item {\bf Vanilla}. 
  We use the term ``vanilla'' to denote our baseline algorithm without any learning component.
  For (proximal) gradient descent, this corresponds to Nesterov's method cold-started at the zero vector.
  For ADMM-based solvers, this corresponds to the solver cold-started at the zero vector where the hyperparameters are set to match the standard, non-relaxed ADMM algorithm.
  \item {\bf Nearest-neighbor warm start}. The nearest-neighbor method (used in~\citet{l2ws}) uses the same vanilla algorithm but warm-starts the new problem at the solution of the nearest training problem measured in terms of parameter distance.
  In all of our examples, this approach does not bring significant improvements. 
  We believe that this is because the problem parameters are too far apart.
  \item {\bf Backtracking line search}. For (proximal) gradient descent, we compare against a backtracking line search, a popular adaptive method.
  The computational burden of each iteration is larger than that for all other methods.
  \item {\bf Learned algorithm hyperparameters (LAH)~\citep{learn_algo_steps}}.
  This approach is the earlier work that we build off, which learns step sizes but does not include momentum terms.  
  Worst-case guarantees are not considered in this work.
  Empirically, we see that in some examples, this can lead to poor behavior.
  \item {\bf Learned metrics (LM)~\citep{metric_learning}}.
  This method uses a neural network to map the problem parameter $x$ to a \emph{metric} for operator-splitting algorithms like ADMM.
  We compare against it in our ADMM-based experiments.
  We use a neural network with five layers of $400$ nodes each as noted in~\citet{metric_learning}. %
  \item {\bf Learned warm starts (L2WS)~\citep{l2ws}}.
  This approach uses a neural network to predict a warm start from the parameter to improve performance after some pre-defined number of iterations.
  We use a neural network with two layers of $500$ nodes each and train on the distance to optimality as these choices generally yield the strongest results~\citep{l2ws}.
\end{itemize}

\begin{table}[!h]
  \small
  \centering
\renewcommand*{\arraystretch}{1.0}
\caption{Overview of the numerical examples.
In each example, we specify the base algorithm to which our learned acceleration approach is applied.
The type of worst-case guarantee we get depends on the algorithm and the structure of problem (see \Sec~\ref{subsec:eval_robustness}).
}
\label{table:num_weights}
\adjustbox{max width=\textwidth}{
\begin{tabular}{lllll}
\toprule
~ \begin{tabular}{@{}c@{}}numerical example ~\end{tabular} 
& \begin{tabular}{@{}c@{}}base algorithm ~\end{tabular} 
& \begin{tabular}{@{}c@{}}parameter  \\size $d$ \end{tabular} 
& \begin{tabular}{@{}c@{}}fixed-point \\variables $\fplen$  ~\end{tabular} 
& \begin{tabular}{@{}c@{}}worst-case guarantee \end{tabular}  \\
\midrule
\begin{filecontents*}{data/num_weights.csv}
problem;algo;d;n;metric
logistic regression;gradient descent;157,000;785;function suboptimality
sparse coding;proximal gradient descent;250;500;function suboptimality
non-negative least squares;proximal gradient descent;700;500;distance to optimal solution
quadcopter control;OSQP (ADMM);44;1,750;fixed-point residual
robust Kalman filtering;SCS (ADMM);100;1,150;fixed-point residual
\end{filecontents*}
\csvreader[head to column names, /csv/separator=semicolon, late after line=\\]{data/num_weights.csv}{
problem=\colA,
algo=\colB,
d=\colC,
n=\colD,
metric=\colE
}{\colA & \colB & \colC & \colD & \colE }
\bottomrule
\end{tabular}
}
\end{table}

\subsection{Accelerated gradient descent}\label{subsec:num_gd}
In this subsection, we apply our method to learn the hyperparameters for accelerated gradient descent.
We demonstrate its effectiveness for logistic regression in \Subsec~\ref{subsubsec:logistic}.

\subsubsection{Logistic regression}\label{subsubsec:logistic}
We study the standard binary classification task of logistic regression. 
The dataset consists of feature vectors $\{a_j\}_{j=1}^m$ where each $a_j \in \reals^q$, and corresponding binary labels $\{l_j\}_{j=1}^m$ where each $l_j \in \{0, 1\}$.
This learning problem is formulated as the convex optimization problem
\begin{equation}
  \begin{array}{ll}
  \label{prob:logisticgd_example}
  \mbox{minimize} & (1 / m) \sum_{j=1}^m l_j \log \left( \sigma(w^T a_j + b) \right) + (1 - l_j) \log \left(1 - \sigma(w^T a_j + b) \right),
  \end{array}
\end{equation}
where $\sigma : \reals \to (0,1)$ denotes the sigmoid function, defined as $\sigma(z) = 1 / (1 + \exp(-z))$. 
The decision variables are the weight vector $w \in \reals^q$ and bias term $b \in \reals$.
In our formulation, the problem parameter is the dataset: $x = (a_1, \dots, a_m, l_1, \dots, l_m) \in \reals^{(q+1)m}$.
An upper bound on the smoothness of the objective of problem~\eqref{prob:logisticgd_example} can be calculated by taking the maximum of the eigenvalue of $(1/(4m)) A(x)^T A(x)$ where $A(x) \in \reals^{q \times m}$ is formed by stacking the $a_j$ vectors.

\myparagraph{Numerical example}
In our numerical experiment we consider logistic regression problems of classifying MNIST images into two classes.
To do so, for each problem, we randomly select two different classes of digits (from $0$ to $9$) and randomly select $100$ images from each class to form a dataset of $m=200$ data points.
We let $f$ be the objective of~\eqref{prob:logisticgd_example}, and define $\mathcal{X} = \{x \mid (1/(4m)) A(x)^T A(x) \leq L\}$, where the estimated smoothness value $L$ is taken as the largest smoothness value over the training set.
Hence, $(f,\mathcal{X})$ is $\mathcal{F}_{0,L}$-parametrized.
We train for robustness using $\gamma^{\rm target} = 0.2$ and $\gamma^{\rm target} = 0.8$.
We use the backtracking line search method from~\citet[Algorithm~3.1]{NoceWrig06}.

\myparagraph{Results}
We illustrate the effectiveness of our approach in Figure~\ref{fig:logistic_regression_results} and Table~\ref{tab:logistic}.
In this example, there is a major benefit in learning the momentum sizes as opposed to only learning the step sizes.
There is a clear tradeoff between robustness and empirical performance on the problem distribution.

\begin{figure}[!h]
  \centering
  \includegraphics[width=0.85\linewidth]{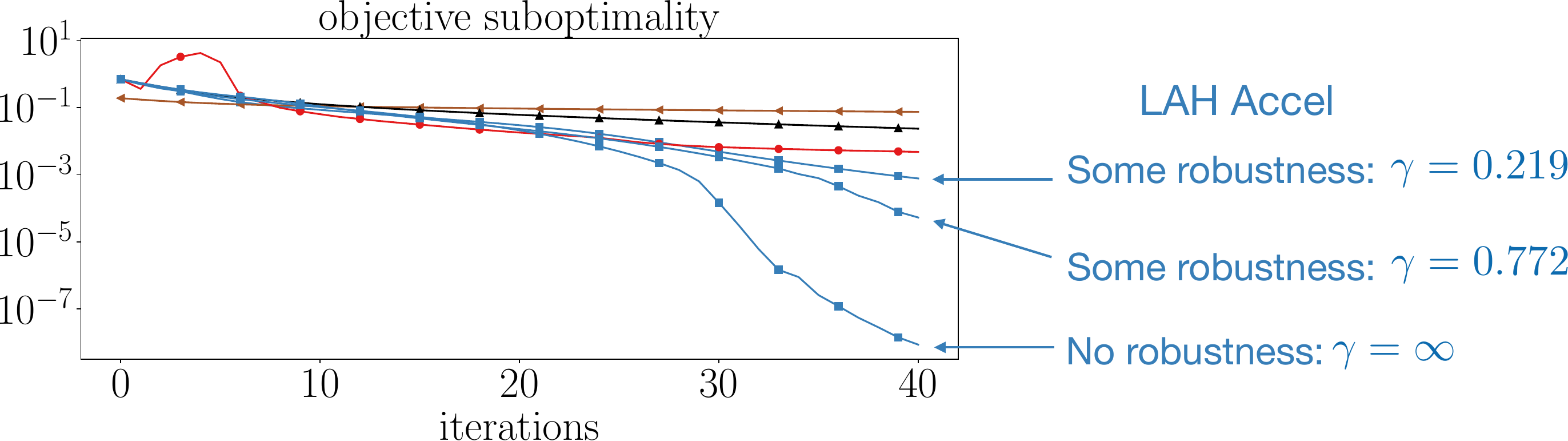}
    \\
    \legendlogistic
    \caption{Logistic regression results.
    The LAH Accel method with no robustness outperforms the LAH method by about $6$ orders of magnitude.
    When trained with robustness, the LAH Accel still outperforms other methods by wide margins.
    }
    \label{fig:logistic_regression_results}
\end{figure}

\begin{table}[!h]
  \centering
  \small
  \renewcommand*{\arraystretch}{1.0}
  \caption{Logistic regression. \itersunconstrained}
  \label{tab:logistic}
  \vspace*{-3mm}
  \adjustbox{max width=\textwidth}{
    \begin{tabular}{llllllllll}
      \toprule
      \begin{tabular}{@{}c@{}}Tol. \end{tabular} &
      \begin{tabular}{@{}c@{}}Nesterov\\$\gamma=0.011$\end{tabular} &
      \begin{tabular}{@{}c@{}}Nearest \\ neighbor\\$N=10k$\end{tabular} &
      \begin{tabular}{@{}c@{}}Backtracking\\line search\end{tabular} &
      \begin{tabular}{@{}c@{}}L2WS\\$N=10$\end{tabular} &
      \begin{tabular}{@{}c@{}}L2WS\\$N=10k$\end{tabular} &
      \begin{tabular}{@{}c@{}}LAH\\$\gamma=\infty$\\$N=10$\end{tabular} &
      \begin{tabular}{@{}c@{}}LAH Accel\\ $\gamma=\infty$\\$N=10$\end{tabular} &
      \begin{tabular}{@{}c@{}}LAH Accel \\ $\gamma=0.219$\\$N=10$\end{tabular} &
      \begin{tabular}{@{}c@{}}LAH Accel \\ $\gamma=0.772$\\$N=10$\end{tabular} \\
      \midrule
      \begin{filecontents*}{./data/logistic_regression/accuracies.csv}
,accuracies,lah_accel_1,backtracking,nearest_neighbor,nesterov,learned_no_accel,lah_accel,lah_accel_2,l2ws,l2ws10000
0,0.1,9,{\bf 4},24,12,8,11,11,15,18
1,0.01,27,93,56,60,24,{\bf 23},26,61,44
2,0.001,39,963,135,183,3747,{\bf 29},35,193,139
3,0.0001,101,6718,418,484,9374,{\bf 31},39,508,462

\end{filecontents*}
\csvreader[
        head to column names,
        late after line=\\
      ]{./data/logistic_regression/accuracies.csv}{
        accuracies=\colA,
        nesterov=\colB,
        nearest_neighbor=\colC,
        backtracking=\colD,
        l2ws=\colE,
        l2ws10000=\colF,
        learned_no_accel=\colG,
        lah_accel=\colH,
        lah_accel_1=\colI,
        lah_accel_2=\colJ
      }{
        \colA & \colB & \colC & \colD & \colE &  \colF & \colG & \colH & \colI & \colJ
      }
      \bottomrule
    \end{tabular}
  }
\end{table}

\subsection{Accelerated proximal gradient descent}\label{subsec:num_prox_gd}
In this subsection, we use our method to learn the hyperparameters for accelerated proximal gradient descent.
We demonstrate its effectiveness for sparse coding in \Subsec~\ref{subsubsec:lasso} and non-negative least squares in \Subsec~\ref{subsubsec:nonneg_ls}.

\subsubsection{Sparse coding}\label{subsubsec:lasso}
In sparse coding, the goal is to reconstruct an input from a sparse linear combination of bases.
A standard approach is to solve the lasso problem~\citep{tibshirani1996regression}, formulated as
\begin{equation}
  \begin{array}{ll}
  \label{prob:lasso}
  \mbox{minimize} & (1 / 2) \|A z - x\|_2^2 + \nu \|z\|_1,
  \end{array}
\end{equation}
where $z \in \reals^\fplen$ is the decision variable, $A \in \reals^{d \times \fplen}$, and $\nu \in \reals_{++}$ are problem data fixed for all instances, and $x \in \reals^d$ is the problem parameter.
The proximal operator of the $\ell_1$-norm is $\prox_{\alpha \|\cdot\|_1}(v) = {\bf sign}(v) \max(0, |v| - \alpha)$.

\myparagraph{Numerical example}
To generate a family of lasso problems, we follow the setup in~\citet{alista, l2o}. 
We sample the entries of the matrix $A$ with i.i.d. Guassian entries $\mathcal{N}(0,1/m)$, and then normalize each column of $A$ so that it has norm one.
For each problem instance, we sample a ground truth vector \( z^{\rm true} \in \reals^n \) with i.i.d.\ standard normal entries and randomly set $90\%$ of the entries to zero. 
The observation vector is given by $b = A z^{\rm true} + \epsilon$, where the signal-to-noise ratio is fixed at $40$ dB. 
We set $d = 250$ and $n = 500$ in our example and let $\mathcal{X} = \reals^d$.
We define the smooth part of problem~\eqref{prob:lasso} as $f(z,x) = (1/2) \|A z - x\|_2^2$; the tuple $(f,\mathcal{X})$ is $\mathcal{Q}_{0,L}$-parametrized (from Example~\ref{ex:least_squares}), where $0$ and $L$ are the smallest and largest eigenvalues of $A^T A$.
We include an additional baseline that is specialized for sparse coding: LISTA~\citep{lista} which learns sequences of weight matrices used in proximal gradient descent.
We train our method for robustness using $\gamma^{\rm target} = 0.1$.

In this example, we also illustrate the advantages of robustness against out-of-distribution problem instances.
To do so, we generate a new dataset that is created in the same way except the ground truth vector is drawn i.i.d. from the distribution $\mathcal{N}(0,3)$ (\ie, the variance is tripled), again with $90\%$ of the values set to zero.

\myparagraph{Results}
We show the effectiveness of our approach in Figures~\ref{fig:lasso_results} and~\ref{fig:lasso_ood_results} and Table~\ref{tab:sparse_coding}.
In Figure~\ref{fig:lasso_step_sizes}, we show the learned step sizes and momentum sizes for our method.
We use the backtracking line-search procedure described in \citep[Algorithm 3.5]{scheinberg2014fast}.
We also train LISTA for $30$ iterations and switch to ISTA after so that subsequent steps provably converge. %
Even with $10000$ training instances, we observe a significant gap between the training and test performance for LISTA (which has $7.5$ million weights).
Within $30$ iterations, the LAH Accel method outperforms LISTA by orders of magnitude.
We show that training for robustness is essential for strong performance on out-of-distribution problems in Figure~\ref{fig:lasso_ood_results}.
The non-robust LAH and LAH Accel reach a suboptimality that is $8$ or more orders of magnitude higher than the initial point.

\begin{figure}[!h]
  \centering
  \includegraphics[width=0.85\linewidth]{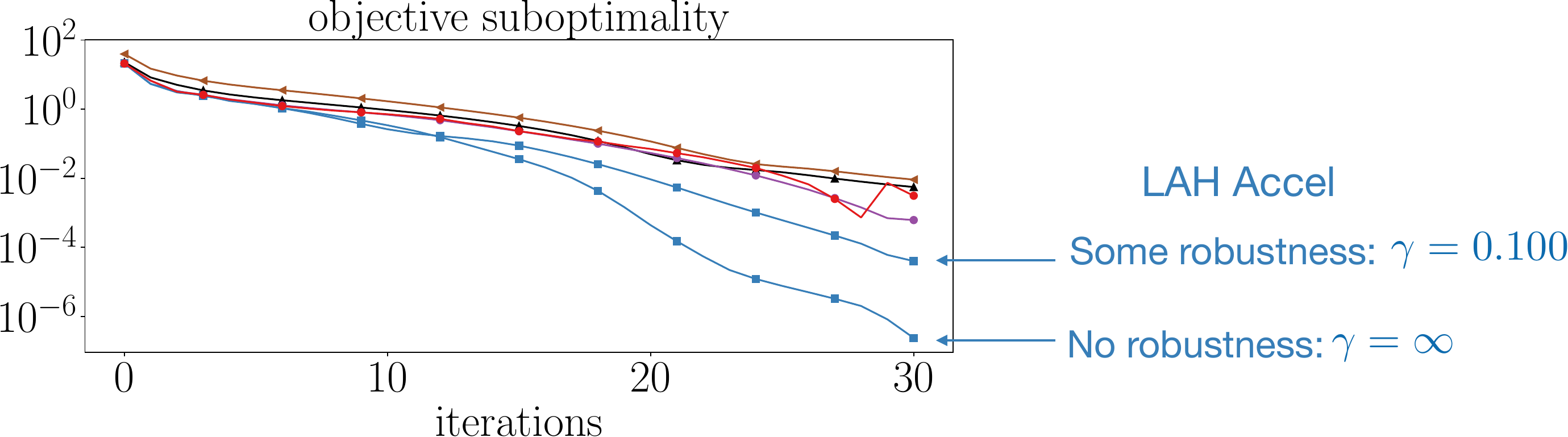}
    \\
    \legendlasso
    \caption{Sparse coding \emph{in-distribution} results.
    The non-robust LAH Accel scheme results in $\gamma=\infty$ and achieves a benefit of about $3$ orders of magnitude after $30$ iterations.
    Here, we pay a relatively small price for robustness since the $\gamma$ value is about $5$ times that of Nesterov's method, but it is still $2$ orders of magnitude better over the parametric family.
    }
    \label{fig:lasso_results}
\end{figure}

\begin{table}[!h]
  \centering
  \small
  \renewcommand*{\arraystretch}{1.0}
  \caption{Sparse coding. \itersunconstrained}
  \label{tab:sparse_coding}
  \vspace*{-3mm}
  \adjustbox{max width=\textwidth}{
    \begin{tabular}{llllllllllll}
      \toprule
      \begin{tabular}{@{}c@{}}Tol. \end{tabular} &
      \begin{tabular}{@{}c@{}}Nesterov\\$\gamma=0.010$\end{tabular} &
      \begin{tabular}{@{}c@{}}Near. \\ neigh. \\ $N=10k$\end{tabular} &
      \begin{tabular}{@{}c@{}}Line \\search\end{tabular} &
      \begin{tabular}{@{}c@{}}L2WS\\$N=10$\end{tabular} &
      \begin{tabular}{@{}c@{}}L2WS\\$N=10k$\end{tabular} &
      \begin{tabular}{@{}c@{}}LISTA\\$N=10$\end{tabular} &
      \begin{tabular}{@{}c@{}}LISTA\\$N=10k$\end{tabular} &
      \begin{tabular}{@{}c@{}}LAH \\$\gamma=\infty$ \\ $N=10$\end{tabular} &
      \begin{tabular}{@{}c@{}}LAH \\guarded \\$N=10$\end{tabular} &
      \begin{tabular}{@{}c@{}}LAH Accel\\ $\gamma=\infty$\\$N=10$\end{tabular} &
      \begin{tabular}{@{}c@{}}LAH Accel \\ $\gamma=0.100$\\$N=10$\end{tabular}
       \\
      \midrule
      \begin{filecontents*}{./data/sparse_coding/accuracies.csv}
,accuracies,lah_accel,nearest_neighbor,nesterov,lah_accel_1,learned_no_accel_safeguard,learned_no_accel,backtracking,lista10000,lista,l2ws,l2ws10000
0,0.1,13,21,19,15,19,19,{\bf 12},28,41,55,18
1,0.01,{\bf 18},30,27,20,25,26,21,38,48,81,25
2,0.001,{\bf 20},44,42,25,29,28,36,44,58,104,39
3,0.0001,{\bf 22},64,60,29,50,36,55,54,74,128,57
\end{filecontents*}
\csvreader[
        head to column names,
        late after line=\\
      ]{./data/sparse_coding/accuracies.csv}{
        accuracies=\colA,
        nesterov=\colB,
        nearest_neighbor=\colC,
        backtracking=\colD,
        l2ws=\colE,
        l2ws10000=\colF,
        lista=\colG,
        lista10000=\colH,
        learned_no_accel=\colI,
        learned_no_accel_safeguard=\colJ,
        lah_accel=\colK,
        lah_accel_1=\colL,
      }{
        \colA & \colB & \colC & \colD & \colE & \colF & \colG & \colH & \colI & \colJ & \colK & \colL
      }
      \bottomrule
    \end{tabular}
  }
\end{table}

\begin{figure}[!h]
  \centering
  \includegraphics[width=0.85\linewidth]{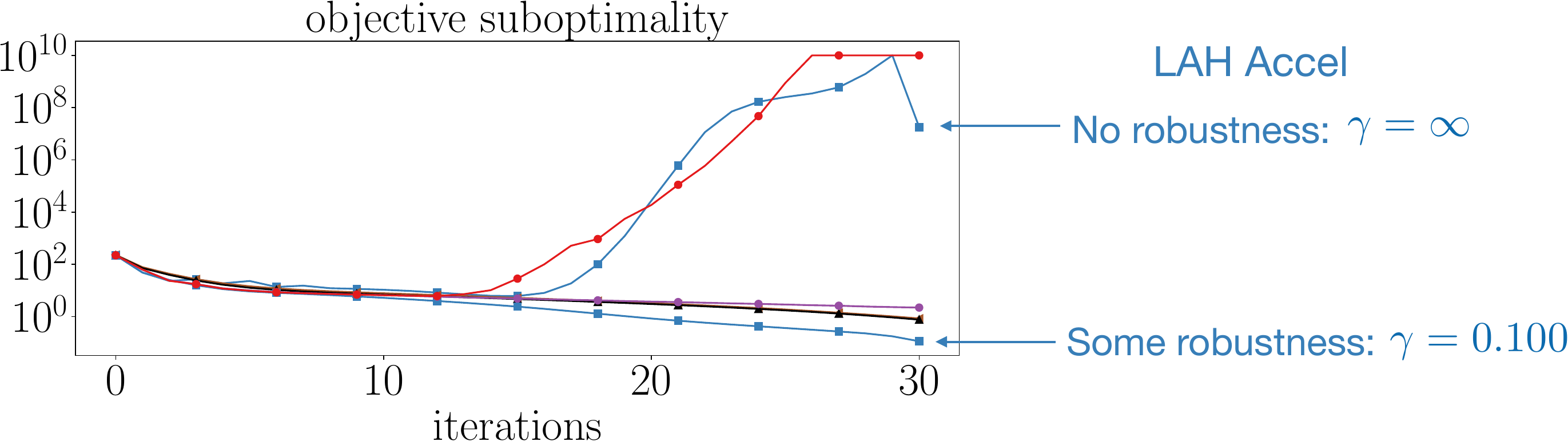}
    \\
    \legendlasso
    \caption{Sparse coding \emph{out-of-distribution} results.
    Both the LAH method (without safeguarding) and the non-robust LAH Accel method diverge. %
    LAH Accel trained with robustness achieves a suboptimality level $6$ times better than Nesterov's method after $30$ iterations.
    }
    \label{fig:lasso_ood_results}
\end{figure}

\begin{figure}[!h]
  \centering
  \includegraphics[width=0.95\linewidth]{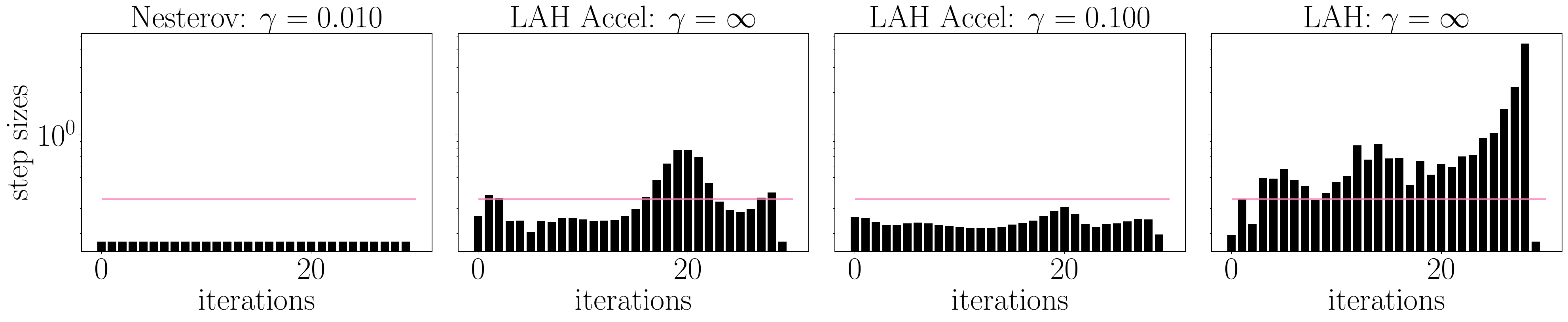}
  \includegraphics[width=0.95\linewidth]{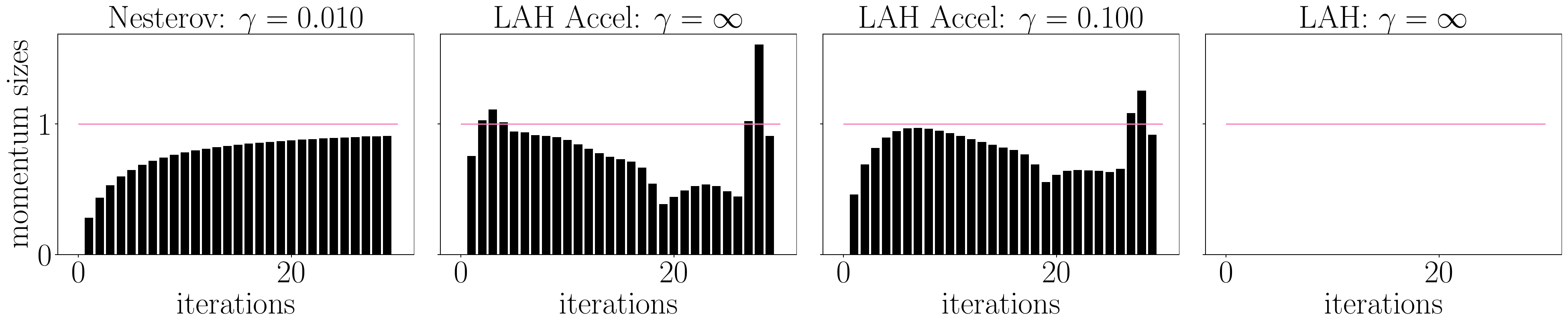}
    \\
    \legendlogisticstep
    \caption{Sparse coding step sizes.
    The y-scale of the step sizes row is logarithmic.
    The LAH Accel column trained with robustness has step sizes all below $2/L$ and all but two momentum sizes below $1$. 
    The non-robust LAH Accel method generally learns larger values.
    Nearly all of the step sizes from LAH are above $(2/L)$, which explains why $\gamma=\infty$ for that method.
    }
    \label{fig:lasso_step_sizes}
\end{figure}

\subsubsection{Non-negative least squares}\label{subsubsec:nonneg_ls}
We now consider the problem of non-negative least squares, formulated as
\begin{equation*}
  \begin{array}{lll}
  \label{prob:nonneg_ls}
  \mbox{minimize} & (1/2)\|Az - x\|_2^2 \quad
  \mbox{subject to} & z \geq 0,
  \end{array}
  \end{equation*}
where $z \in \reals^n$ is the decision variable, $A \in \reals^{d \times n}$ is problem data fixed for all instances, and $x \in \reals^d$ is the problem parameter.

\myparagraph{Numerical example}
We set $d = 700$ and $n = 500$, and create a random $A$ in the same way as we did for sparse coding.
We sample each entry of $x$ i.i.d. from the distribution $\mathcal{N}(0,1)$.
We define $(f,\mathcal{X})$ as in the previous example, which makes the pair $\mathcal{Q}_{\mu,L}$-parametrized, where $\mu$ and $L$ are the smallest and largest eigenvalues of $A^T A$.
We train for robustness using $\gamma^{\rm target} = 0.7$.

\myparagraph{Results}
We showcase the effectiveness of our approach in Figure~\ref{fig:nonneg_ls_results} and Table~\ref{tab:nonneg_ls}.
Again, the LAH Accel method with no robustness gives the strongest empirical performance over the parametric family.
When trained for robustness, LAH Accel about performs as well as LAH---but the latter has no robustness guarantee (\ie, $\gamma=\infty$).

\begin{figure}[!h]
  \centering
  \includegraphics[width=0.85\linewidth]{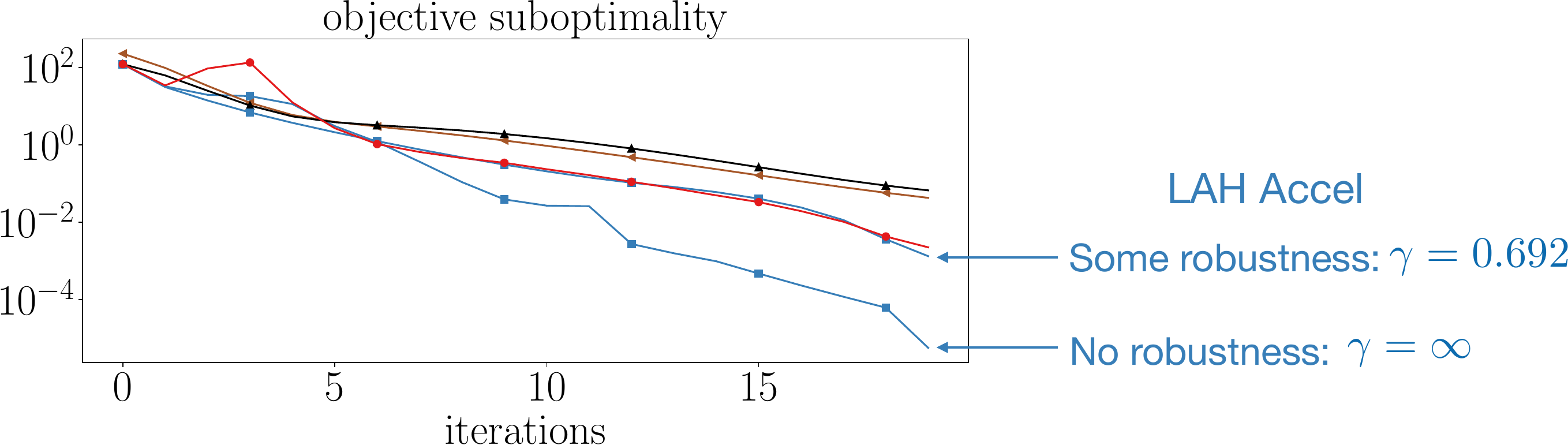}
    \\
    \legendnonnegls
    \caption{Non-negative least squares results.
    The non-robust LAH Accel method performs best.
    The robust LAH Accel method performs similarly to LAH, but LAH is not robust.
    }
    \label{fig:nonneg_ls_results}
\end{figure}

\begin{table}[!h]
  \centering
  \small
  \renewcommand*{\arraystretch}{1.0}
  \caption{Non-negative least squares. \itersunconstrained}
  \label{tab:nonneg_ls}
  \vspace*{-3mm}
  \adjustbox{max width=\textwidth}{
    \begin{tabular}{lllllllll}
      \toprule
      \begin{tabular}{@{}c@{}}Tol. \end{tabular} &
      \begin{tabular}{@{}c@{}}Nesterov \\$\gamma=3.06$\end{tabular} &
      \begin{tabular}{@{}c@{}}Near. \\ neigh. \\ $N=10k$\end{tabular} &
      \begin{tabular}{@{}c@{}}Line \\search\end{tabular} &
      \begin{tabular}{@{}c@{}}L2WS\\$N=10$\end{tabular} &
      \begin{tabular}{@{}c@{}}L2WS\\$N=10k$\end{tabular} &
      \begin{tabular}{@{}c@{}}LAH \\ $\gamma=\infty$\\$N=10$\end{tabular} &
      \begin{tabular}{@{}c@{}}LAH Accel\\ $\gamma=\infty$\\$N=10$\end{tabular} &
      \begin{tabular}{@{}c@{}}LAH Accel \\ $\gamma=0.692$\\$N=10$\end{tabular}
       \\
      \midrule
      \begin{filecontents*}{./data/nonneg_ls/accuracies.csv}
,accuracies,lah_accel,nearest_neighbor,nesterov,lah_accel_1,learned_no_accel,backtracking,l2ws,l2ws10000
0,0.1,{\bf 9},17,18,13,13,13,18,15
1,0.01,{\bf 12},25,28,18,18,24,27,22
2,0.001,{\bf 14},33,36,27,22,37,35,31
3,0.0001,{\bf 18},43,45,36,31,56,44,40
\end{filecontents*}
\csvreader[
        head to column names,
        late after line=\\
      ]{./data/nonneg_ls/accuracies.csv}{
        accuracies=\colA,
        nesterov=\colB,
        nearest_neighbor=\colC,
        backtracking=\colD,
        l2ws=\colE,
        l2ws10000=\colF,
        learned_no_accel=\colI,
        lah_accel=\colK,
        lah_accel_1=\colL,
      }{
        \colA & \colB & \colC & \colD & \colE & \colF & \colI & \colK & \colL
      }
      \bottomrule
    \end{tabular}
  }
\end{table}

\subsection{Accelerated OSQP}\label{subsec:num_osqp}
In this subsection, we apply our method to learn the hyperparameters of an accelerated version of OSQP~\citep{osqp} which is based on ADMM.
In \Subsec~\ref{subsubsec:quadcopter}, we focus on the task of controlling a quadcopter.

\subsubsection{Model predictive control of a quadcopter}\label{subsubsec:quadcopter}
We now consider the task of controlling a quadcopter to follow a reference trajectory using MPC~\citep{borrelli_mpc_book}.
In MPC, we optimize over a finite horizon but apply only the first control input, then re-solve the problem with a new initial state.
We model the quadcopter as a rigid body controlled by four motors using quaternions~\citep{quadcopter_dynamics_mpc}.
The state of the quadcopter $s_t \in \reals^{n_s}$ consists of the position, velocity, and quaternions.
The control of the quadcopter $u_t \in \reals^{n_s}$ consists of vertical thrust and the angular velocities of the body frame.
The non-linear dynamics of the quadcopter can be found in~\citep[\Sec~6.3.1]{l2ws}.
Since the dynamics are non-linear (thus rendering the MPC problem to be non-convex) we linearize them around the current state $s_0$ and the most recent control input denoted as $u_{-1}$~\citep{nonlinear_mpc}.
This yields a linearized dynamics of the form $s_{t+1} \approx A s_t + B u_t$ where $A \in \reals^{n_s \times n_s}$ and $B \in \reals^{n_s \times n_u}$.
At each time step, we aim to track a reference trajectory given by $s^{\rm{ref}} = (s^{\rm{ref}}_1, \dots, s^{\rm{ref}}_T)$, while satisfying constraints on the states and the controls.
We solve the following quadratic program in each timestep:
\begin{equation*}
  \begin{array}{ll}
  \label{eq:quadcopter_qp}
  \mbox{minimize} & (s_{T} - s_T^{\rm{ref}})^T Q_{T} (s_{T} - s_T^{\rm{ref}}) + \sum_{t=1}^{T-1} (s_{t} - s_t^{\rm{ref}})^T Q (s_{t} - s_t^{\rm{ref}}) + u_t^T R u_t \\
  \mbox{subject to} & s_{t+1} = A s_t + B u_t, \quad t=0, \dots, T-1 \\
  & u_{\textrm{min}} \leq u_t \leq u_{\textrm{min}}, |u_{t+1} - u_t| \leq \Delta u, \quad t=0, \dots, T-1 \\
  & s_{\textrm{min}} \leq s_t \leq s_{\textrm{max}}, \quad t=1, \dots, T.
  \end{array}
\end{equation*}
Here, the decision variables are the states $(s_1, \dots, s_{T})$ where $s_t \in \reals^{n_s}$ and the controls $(u_0, \dots, u_{T-1})$ where $u_t \in \reals^{n_u}$.
The problem parameter is $x = (s_0, u_{-1}, s_1^{\rm ref}, \dots, s_T^{\rm ref})$.

\myparagraph{Numerical example}
In this example, we follow the setup from~\citet{l2ws}.
In this example, the problem matrix change across instances (see the OSQP row in Table~\ref{table:fp_algorithms}); the LAH approach is not suitable since a new factorization in each iteration is required~\citep[\Sec~4]{learn_algo_steps}.
Since our problems are sequential, we instead compare against the \emph{previous-solution warm start} method, which warm starts the current problem with the solution of the previous one shifted by one time index~\citep{nonlinear_mpc}.
We let $\mathcal{X} = \reals^d$.
The underlying fixed-point operator for OSQP is $R(\theta^{\rm inv})$-nonexpansive (see Proposition~\ref{prop:admm_nonexp}) where $\theta^{\rm inv}$ is learned.
We train for robustness using $\gamma^{\rm target} = 0.1$.

\myparagraph{Results}
We showcase the effectiveness of our method in Figure \ref{fig:quadcopter_results} and Table~\ref{tab:quadcopter}.
Training with robustness yields solutions of similar quality, but $21$ times better in terms of robustness.
In this example, we visualize the effectiveness of our approach by implementing the control of the quadcopter in closed-loop, where in each iteration, there is a strict budget of iterations.

\begin{figure}[!h]
  \centering
  \includegraphics[width=0.85\linewidth]{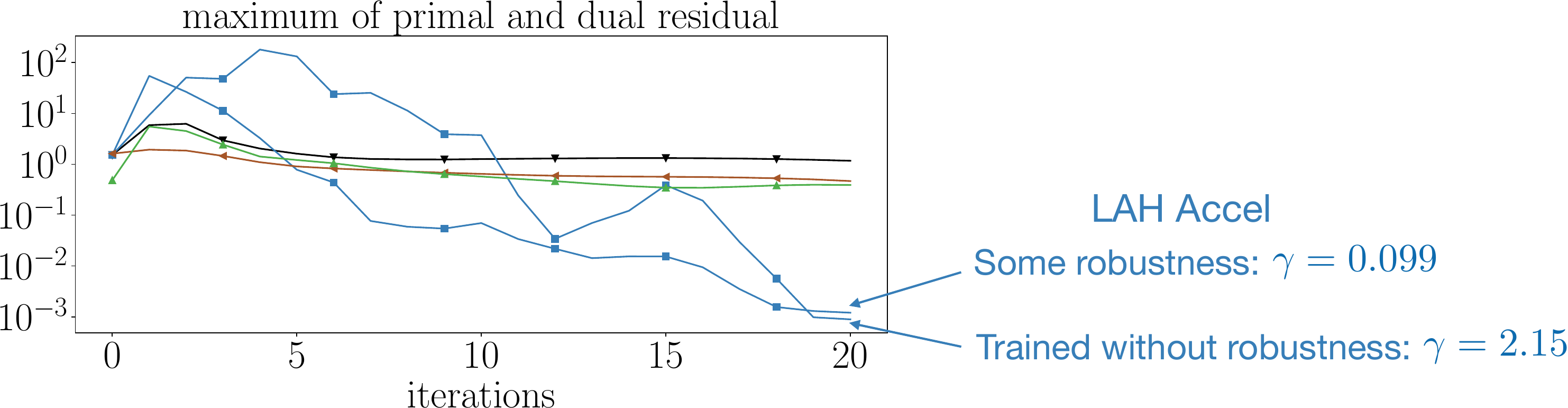}
    \\
    \legendquad
    \caption{Quadcopter results.
    Both LAH Accel methods outperform the others by a wide margin, with the robust variant performing nearly as well as the non-robust one.
    }
    \label{fig:quadcopter_results}
\end{figure}

\begin{table}[!h]
  \centering
  \small
  \renewcommand*{\arraystretch}{1.0}
  \caption{Quadcopter. \iters}
  \label{tab:quadcopter}
  \vspace*{-3mm}
  \adjustbox{max width=\textwidth}{
    \begin{tabular}{lllllllllll}
      \toprule
      \begin{tabular}{@{}c@{}}Tol. \end{tabular} &
      \begin{tabular}{@{}c@{}}Cold \\start\\$\gamma=0.071$\end{tabular} &
      \begin{tabular}{@{}c@{}}Nearest \\ neighbor\\$N=10k$\end{tabular} &
      \begin{tabular}{@{}c@{}}Prev.\\sol.\end{tabular} &
      \begin{tabular}{@{}c@{}}L2WS \\$N=10$\end{tabular} &
      \begin{tabular}{@{}c@{}}L2WS \\$N=10k$\end{tabular} &
      \begin{tabular}{@{}c@{}}LM\\ $N=10$\end{tabular} &
      \begin{tabular}{@{}c@{}}LM \\ $N=10k$\end{tabular} &
      \begin{tabular}{@{}c@{}}LAH Accel\\$\gamma=0.099$\\$N=10$\end{tabular} &
      \begin{tabular}{@{}c@{}}LAH Accel\\$\gamma=2.15$\\$N=10$\end{tabular}\\
      \midrule
      \begin{filecontents*}{./data/quadcopter/accuracies.csv}
,accuracies,cold_start,lah_accel,lah_accel_1,nearest_neighbor,l2ws,l2ws10000,lm10000,lm,prev_sol
0,0.1,301,12,{\bf 7},192,171,114,$>$5000,$>$5000,51
1,0.01,1113,18,{\bf 16},1012,786,667,$>$5000,$>$5000,552
2,0.001,3532,{\bf 19},26,2118,1906,1784,$>$5000,$>$5000,1560
\end{filecontents*}
\csvreader[
        head to column names,
        late after line=\\
      ]{./data/quadcopter/accuracies.csv}{
        accuracies=\colA,
        cold_start=\colB,
        nearest_neighbor=\colC,
        prev_sol=\colD,
        l2ws=\colE,
        l2ws10000=\colF,
        lm=\colG,
        lm10000=\colH,
        lah_accel=\colJ,
        lah_accel_1=\colK
      }{
        \colA & \colB & \colC & \colD & \colE & \colF & \colG & \colH & \colJ & \colK
      }
      \bottomrule
    \end{tabular}
  }
\end{table}

\begin{figure}[!h]
  \centering
  \begin{subfigure}[t]{0.18\linewidth}
    \centering
    \includegraphics[width=1\textwidth]{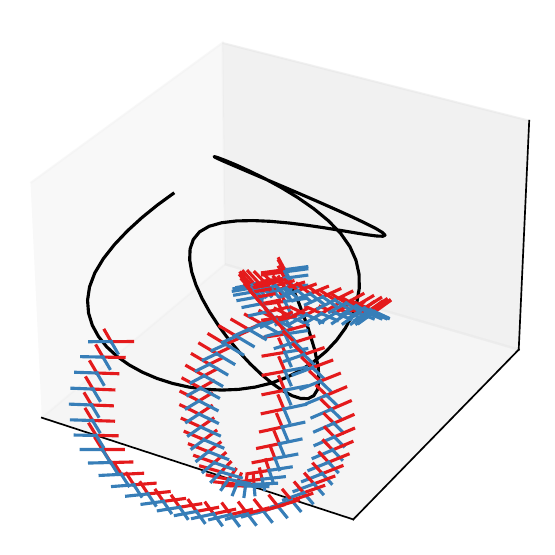}
    \label{fig:quad_prev_sol2}
  \end{subfigure}%
  \hspace*{4mm}
  \begin{subfigure}[t]{0.18\linewidth}
    \centering
    \includegraphics[width=1\linewidth]{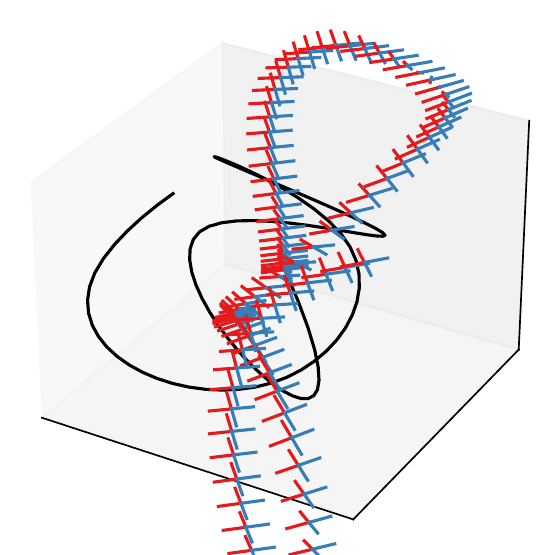}
    \label{fig:quad_nearest_neighbor2}
  \end{subfigure}
  \hspace*{4mm}
  \begin{subfigure}[t]{0.18\linewidth}
    \centering
    \includegraphics[width=1\linewidth]{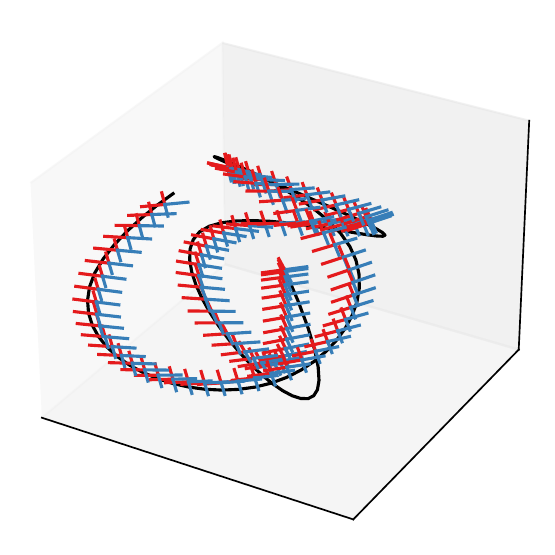}
    \label{fig:quad_k5_reg2}
  \end{subfigure}\\[-20pt]
  \begin{subfigure}[t]{0.18\linewidth}
    \centering
    \includegraphics[width=1\textwidth]{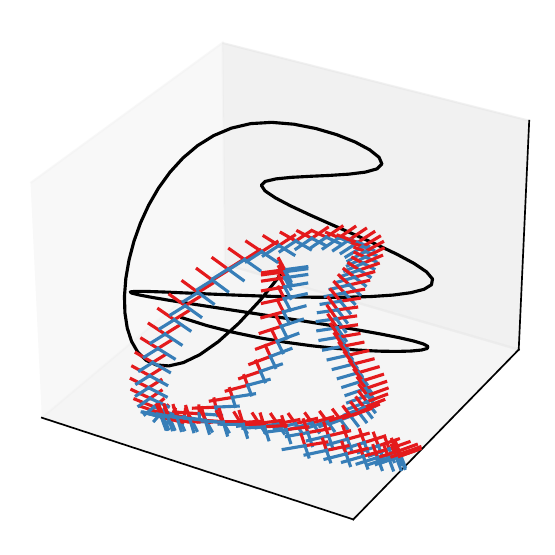}
    \caption{\centering{previous solution $80$ iters}}
    \label{fig:quad_prev_sol3}
  \end{subfigure}%
  \hspace*{4mm}
  \begin{subfigure}[t]{0.18\linewidth}
    \centering
    \includegraphics[width=1\linewidth]{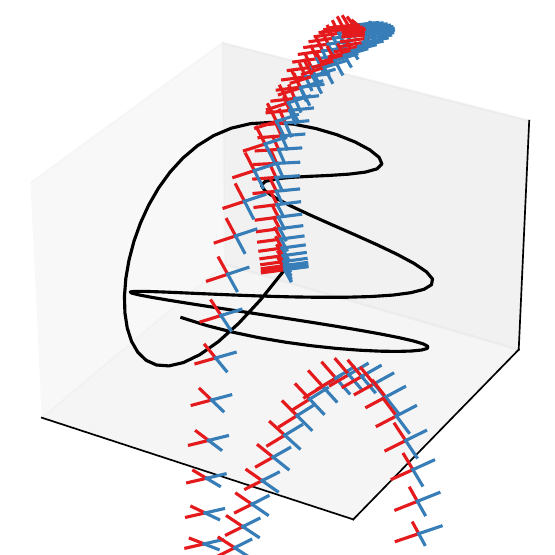}
    \caption{\centering{nearest neighbor $80$ iters}}
    \label{fig:quad_nearest_neighbor}
  \end{subfigure}
  \hspace*{4mm}
  \begin{subfigure}[t]{0.18\linewidth}
    \centering
    \includegraphics[width=1\linewidth]{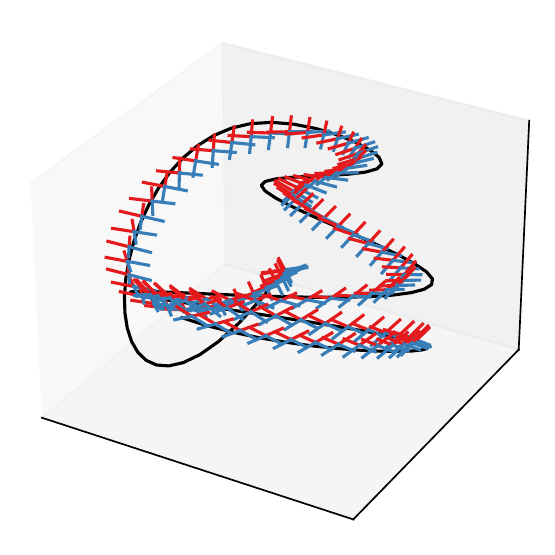}
    \caption{\centering{LAH Accel $20$ iters}}
    \label{fig:quad_k5_reg3}
  \end{subfigure}
  \caption{Quadcopter visualizations.
    Each row corresponds to a different unseen trajectory.
    The LAH Accel approach is able to track the reference trajectory fairly well.
    Even with $4$ times as many iterations, the other methods fail to track the trajectory.
  }
  \label{fig:quadcopter_visualize}
\end{figure}

\subsection{Accelerated SCS}\label{subsec:num_scs}
In this subsection, we apply our method to learn the hyperparameters of an accelerated version of SCS~\citep{scs_quadratic} which is based on ADMM.
In \Subsec~\ref{subsubsec:rkf}, we apply our method to the task of robust Kalman filtering.

\subsubsection{Robust Kalman filtering}\label{subsubsec:rkf}
We now consider the task of tracking an autonomous vehicle provided noisy measurement data.
This can be modeled as a linear dynamical system given by the equations $s_{t+1} = A s_t + B u_t, y_{t} = C s_t + v_t$
for $t=0,1,\dots$ where $s_t \in \reals^{n_s}$ gives the state of the system at time $t$, $y_t \in \reals^{n_o}$ gives the observations at time $t$, and $u_t \in \reals^{n_u}$ and $v_t \in \reals^{n_o}$ are noise injected into the system.
The system matrices are $A \in \reals^{n_s \times n_s}$, $B \in \reals^{n_s \times n_u}$, and $C \in \reals^{n_o \times n_s}$.
Kalman filtering estimates states $\{s_t\}$ given noisy measurements $\{y_t\}$. 
The standard assumption of Gaussian i.i.d. noise, however, can degrade its performance under outliers~\citep{rkf}. 
Robust Kalman filtering addresses this issue, which can be formulated as
\begin{equation}
  \begin{array}{ll}
  \label{prob:rkf}
  \mbox{minimize} & \sum_{t=1}^{T-1} \|u_t\|_2^2 + \nu \psi_{\rho} (v_t) \\
  \mbox{subject to} & s_{t+1} = A s_t + B u_t, \quad  y_t = C s_t + v_t, \quad t=0,\dots, T-1.
  \end{array}
\end{equation}
Here, the Huber penalty function~\citep{huber} parametrized by $\rho \in \reals_{++}$ 
is given as $\psi_{\rho}(a) = \|a\|_2^2$ if $\|a\|_2 \leq \rho$ and $2 \rho \|a\|_2 - \rho^2$ otherwise; smaller $\rho$ implies more robustness to outliers. 
The problem parameter is given by the measurements $x=(y_0,\dots,y_{T-1})$.
Problem~\eqref{prob:rkf} can be formulated as a second-order cone program (SOCP)~\citep{neural_fp_accel_amos}.

\myparagraph{Numerical example}
We follow the numerical setup used in~\citet{neural_fp_accel_amos,l2ws,learn_algo_steps} to track a vehicle moving in two-dimensional space.
The state and input sizes are $n_s=4$ and $n_u=2$.
We refer the reader to~\citep[\Sec~6.4.1]{l2ws} for the values of $A$, $B$, and $C$.
We let $\mathcal{X} = \reals^d$.
The underlying fixed-point operator for OSQP is $R(\theta^{\rm inv})$-nonexpansive (see Proposition~\ref{prop:admm_nonexp}) where $\theta^{\rm inv}$ is learned.
We train for robustness using $\gamma^{\rm target} = 0.1$.

\myparagraph{Results}
We showcase the effectiveness of our approach in Figure~\ref{fig:rkf_results} and Table~\ref{tab:rkf}.
Since we solve these SOCPs in sequence, we compare against the previous-solution warm start method described in \Subsec~\ref{subsubsec:quadcopter}.
The LAH Accel methods dramatically outperforms all other methods.
In this case, both of robustness values are below the target values. %
We visualize the benefits of the learning acceleration approach in Figure~\ref{fig:rkf_visuals}.
To better visualize differences between methods, we reduce the iteration budget to $5$ (and hence, train using $K=5$) and repeat the learning process.

\begin{figure}[!h]
  \centering
  \includegraphics[width=0.85\linewidth]{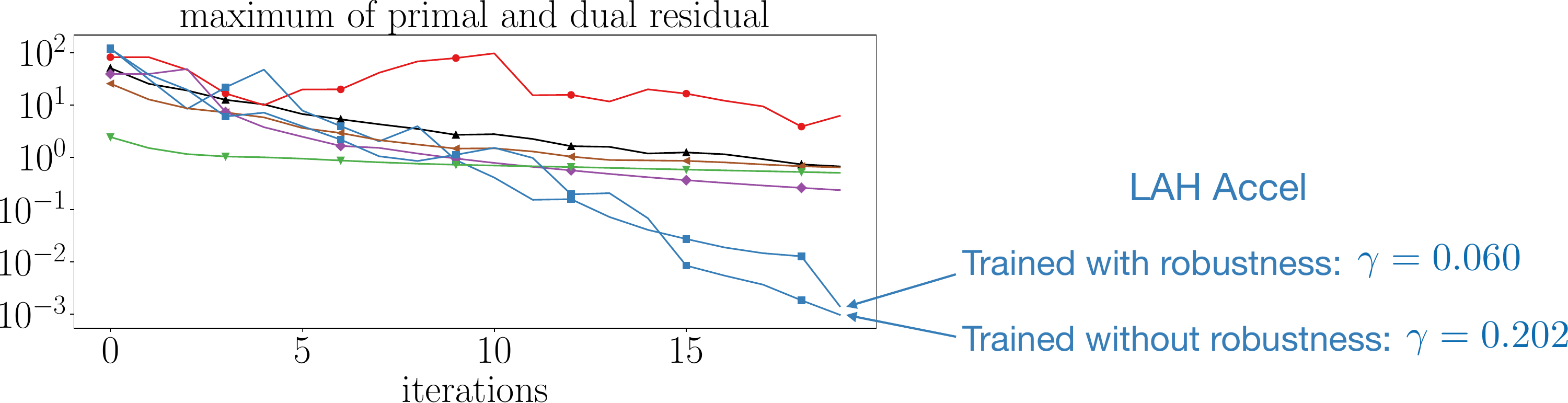}
    \\
    \legendrkf
    \caption{Robust Kalman filtering results.
    Both LAH Accel methods dramatically outperform the other methods while also achieving robustness.
    }
    \label{fig:rkf_results}
\end{figure}

\begin{table}[!h]
  \centering
  \small
  \renewcommand*{\arraystretch}{1.0}
  \caption{Robust Kalman filtering. \iters}
  \label{tab:rkf}
  \vspace*{-3mm}
  \adjustbox{max width=\textwidth}{
    \begin{tabular}{lllllllllll}
      \toprule
      \begin{tabular}{@{}c@{}}Tol. \end{tabular} &
      \begin{tabular}{@{}c@{}}Cold \\start\\$\gamma=0.071$\end{tabular} &
      \begin{tabular}{@{}c@{}}Nearest \\ neighbor\\$N=10k$\end{tabular} &
      \begin{tabular}{@{}c@{}}Prev.\\sol.\end{tabular} &
      \begin{tabular}{@{}c@{}}L2WS \\$N=10$\end{tabular} &
      \begin{tabular}{@{}c@{}}L2WS \\$N=10k$\end{tabular} &
      \begin{tabular}{@{}c@{}}LM\\ $N=10$\end{tabular} &
      \begin{tabular}{@{}c@{}}LM \\ $N=10k$\end{tabular} &
      \begin{tabular}{@{}c@{}}LAH\\$N=10$\end{tabular} &
      \begin{tabular}{@{}c@{}}LAH Accel\\$\gamma=0.060$\\$N=10$\end{tabular} &
      \begin{tabular}{@{}c@{}}LAH Accel\\$\gamma=0.202$\\$N=10$\end{tabular}\\
      \midrule
      \begin{filecontents*}{./data/robust_kalman/accuracies.csv}
,accuracies,cold_start,learned_no_accel,lah_accel,nearest_neighbor,lm10000,lm,prev_sol,lah_accel_1,l2ws10000,l2ws
0,0.1,54,21,{\bf 13},69,32,39,70,14,65,62
1,0.01,137,61,19,152,85,82,154,{\bf 15},143,141
2,0.001,231,86,21,246,146,131,248,{\bf 19},234,231
3,0.0001,328,121,{\bf 30},343,211,187,345,31,329,326
\end{filecontents*}
\csvreader[
        head to column names,
        late after line=\\
      ]{./data/robust_kalman/accuracies.csv}{
        accuracies=\colA,
        cold_start=\colB,
        nearest_neighbor=\colC,
        prev_sol=\colD,
        l2ws=\colE,
        l2ws10000=\colF,
        lm=\colG,
        lm10000=\colH,
        learned_no_accel=\colI,
        lah_accel=\colJ,
        lah_accel_1=\colK
      }{
        \colA & \colB & \colC & \colD & \colE & \colF & \colG & \colH & \colI & \colJ & \colK
      }
      \bottomrule
    \end{tabular}
  }
\end{table}

\begin{figure}[!h]
  \centering
  \includegraphics[width=0.23\linewidth]{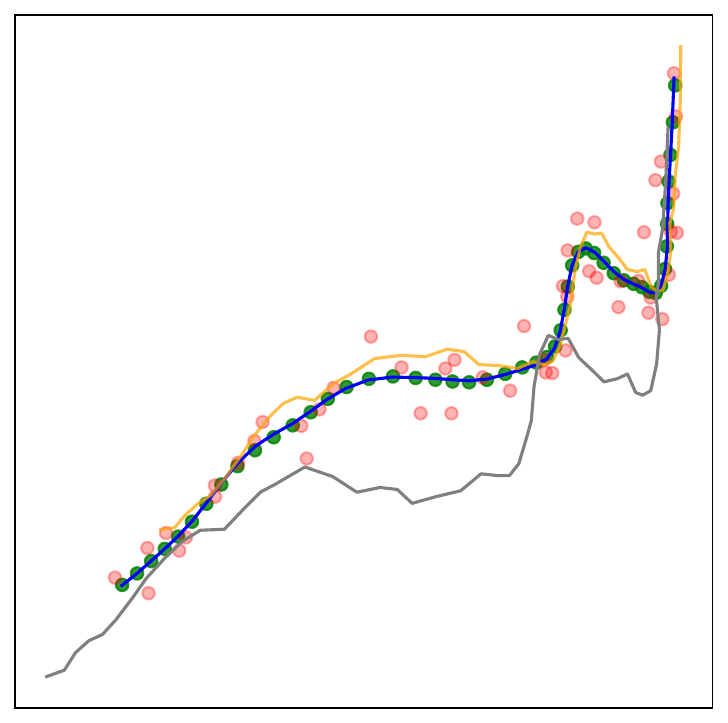}
  \includegraphics[width=0.23\linewidth]{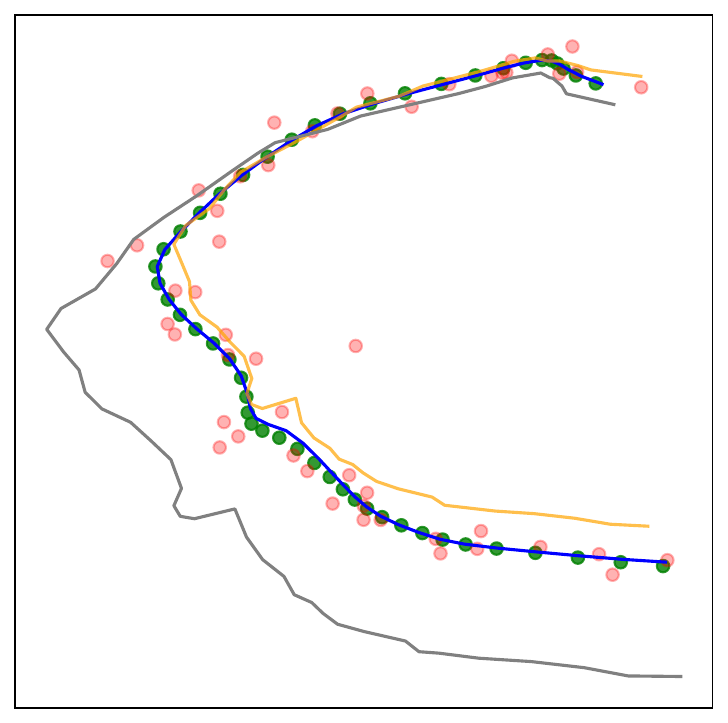}
  \raisebox{0pt}{\rule{4pt}{85pt}}
  \includegraphics[width=0.23\linewidth]{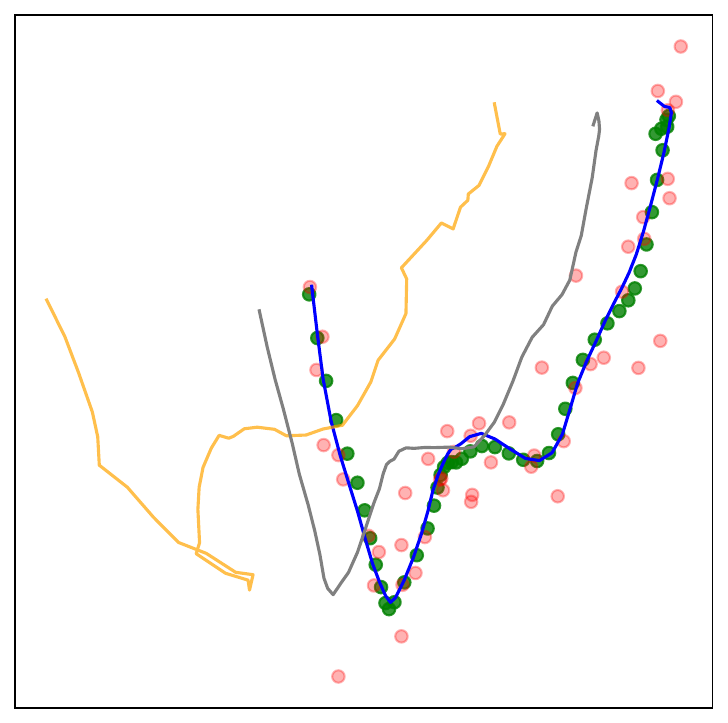}
  \includegraphics[width=0.23\linewidth]{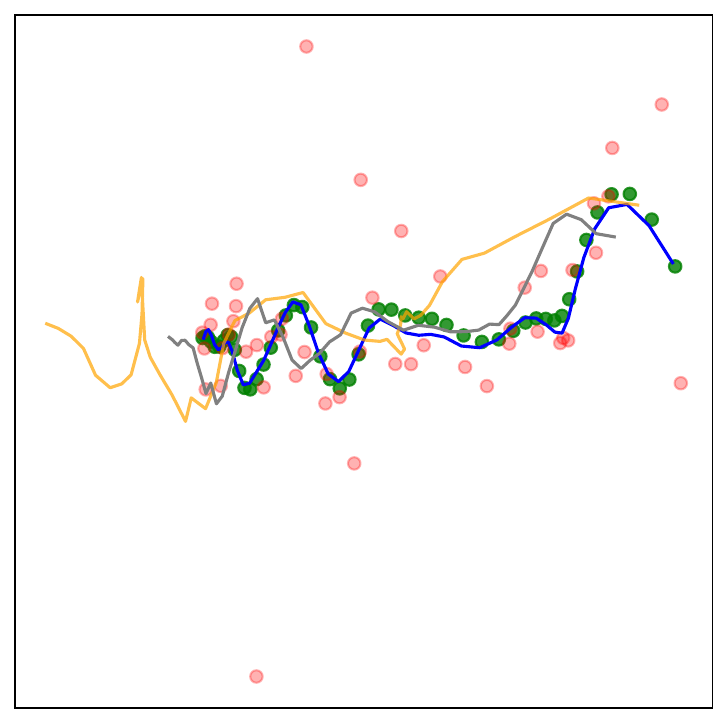}
    \rkfvisualslegend
    \caption{Robust Kalman filtering visualizations after $5$ iterations.
    The noisy trajectory is given by the pink dots and the optimal solution of the SOCP is given by the green dots.
    {\bf Left: two problem instances from in-distribution rollouts.}
    There is a clear hierarchy in performance: LAH Accel takes the lead, followed by LAH, then no learning.
    {\bf Right: two problem instances from out-of-distribution rollouts.}
    To generate these problems, we repeat the process, but initialize the new trajectory at $(100,0)$ instead of at the origin $(0,0)$.
    These problems are equivalent under a simple shift of coordinates.
    While the LAH Accel approach still performs the best, LAH is clearly worse than the vanilla approach. 
    }
    \label{fig:rkf_visuals}
\end{figure}

\section{Conclusion}\label{sec:conclusion}
We develop a framework to learn the acceleration hyperparameters for parametric convex optimization problems while certifying robustness.
In particular, we incorporate PEP for the hyperparameters to achieve a desired level of worst-case guarantee over all parameters of interest.
In our numerical examples, our approach dramatically improves the quality of the solution within a budget of iterations and guarantees worst-case performance for all parameters in a set. %

As a continuation in the line of work on learning to optimize \emph{only} a few algorithm hyperparameters~\citep{learn_algo_steps}, we expand upon the previous work by learning in addition the momentum terms and refining its computational benefit by introducing the time-invariant hyperparameters (for ADMM-based solvers). 
We further uncover an additional benefit of this approach---that it aligns well with the PEP analysis, which provides worst-case guarantees for given algorithms with arbitrary hyperparameters.

We see several avenues for future work.
First, for our acceleration algorithms, each iterate is a linear combination of the past two iterates; extending the look-back horizon could unlock even greater performance gains.
Second, since the size of the SDP grows with the number of iterations, it becomes computationally challenging to train for a large number of iterations.
Indeed, in our experiments, the largest budget we consider is $40$ steps.
It would be useful to scale the method to longer horizons.
Last, our worst-case guarantees derived from the PEP framework can be overly conservative.
Tighter guarantees may be achievable by combining this approach with the exact worst-case analysis of~\citet{ranjan2024exact}, which provides sharp bounds on performance.
\section*{Acknowledgements} NM is supported in part by NSF Award SLES-2331880, NSF CAREER award ECCS-2045834, and AFOSR Award FA9550-24-1-0102.

\bibliography{bibliographynourl}

\appendix
\section{Details of accelerated first-order methods from Table~\ref{table:fp_algorithms}}\label{sec:fom_details}

\myparagraph{Accelerated gradient descent} %
Here, $z \in \reals^n$ is the decision variable, and $f : \reals^\fplen \times \reals^d \rightarrow \reals$ is an $L$-smooth, convex objective function with respect to $z$.

\myparagraph{Accelerated proximal gradient descent} %
Here, $z \in \reals^n$ is the decision variable, $f : \reals^\fplen \times \reals^d \rightarrow \reals$ is an $L$-smooth, convex function with respect to $z$, and $g : \reals^\fplen \times \reals^d \rightarrow \reals$ is a non-smooth, convex function with respect to $z$.

\myparagraph{Accelerated ADMM}
Here, $w \in \reals^q$ is the decision variable, $f : \reals^q \times \reals^d \rightarrow \reals \cup \{+\infty\}$ is a closed, convex, and proper function with respect to $z$, and $g : \reals^q \times \reals^d \rightarrow \reals \cup \{+\infty\}$ is a non-smooth, but a closed, convex, and proper function with respect to $z$.
Exploiting the established equivalence between ADMM and Douglas-Rachford splitting~\citep{admm_dr_equiv}, we give the Douglas-Rachford iterations in Table~\ref{table:fp_algorithms}.

\myparagraph{Accelerated OSQP}
The problem data is $P \in \symm^p_+$, $A \in \reals^{m \times p}$, $c \in \reals^p$, $l \in \reals^m$, and $u \in \reals^m$.
The operator $\Pi_{[l,u]}(v)$ projects the vector $v$ onto the box $[l,u]$.
We split the vector $\rho \in \reals^m$ into $\rho = (\rho_{\rm eq} \mathbf{1}_{m_{\rm eq}}, \rho_{\rm ineq} \mathbf{1}_{m_{\rm ineq}})$ where $m_{\rm eq}$ is the number of constraints where $l = u$, and $m_{\rm eq}$ is the number of constraints where $l < u$.
The primal and dual solutions at the $k$-th step are given by $w^k$ and $y^k = \rho(v^k - \Pi_{[l,u]}(v^k))$.
The primal and dual residuals are $\|Aw^k - \Pi_{[l,u]}(v^k)\|_2$ and $\|Pw^k + A^T y^k + c\|_2$.

\myparagraph{Accelerated SCS}
The problem data is $P \in \symm^p_+$, $A \in \reals^{m \times p}$, $c \in \reals^p$, and $b \in \reals^m$.
The operator $\Pi_{\reals^p \times \mathcal{K}^*}(\xi)$ projects $\xi$ onto the cone $\reals^p \times \mathcal{K}^*$.
In SCS, positive scalings \(r_w\), \(r_{y_z}\), and \(r_{y_{\mathrm{nz}}}\) apply to the primal variable \(w\) and to the dual variable \(v\) associated with zero- and non-zero-cone constraints, whose counts are \(m_{\mathrm{z}}\) and \(m_{\mathrm{nz}}\).
The primal and dual residuals at the $k$-th iteration are $\|Aw^k + s^k - b\|_2$ and $\|Pw^k + A^T v^k + c\|_2$.
We do not implement the homogeneous self-dual embedding of SCS in~\citep{scs_quadratic} for simplicity.

\section{Operator theory definitions and proofs}\label{sec:op_theory_proof}

\subsection{Operator theory definitions}\label{sec:op_theory}

\begin{definition}[Nonexpansive operator]
  Assume that $R \succ 0$.
  An operator $T$ is nonexpansive with respect to $R$ if $\|Tx - Ty\|_R \leq \|x - y\|_R \quad \forall x, y \in \textbf{dom}(T)$.
\end{definition}

\begin{definition}[$\alpha$-averaged operator]
  Assume that $R \succ 0$.
  An operator $T$ is $\alpha$-averaged with respect to the metric $R$ for $\alpha \in [0, 1)$ if there exists a nonexpansive operator $S$ with respect to $R$ such that $T = (1 - \alpha) I + \alpha S$.
\end{definition}

\begin{definition}[Resolvent operator]
  The resolvent of operator $A$ is $J_A = (I + A)^{-1}$.
\end{definition}

\begin{definition}[Monotone operator]
  An operator $A$ is monotone in $\reals^n$ if $(v - v')^T(u-u') \geq 0$ for all possible $u,u' \in \textbf{dom}(A)$ and $v \in A(u)$ and $v' \in A(u')$.
\end{definition}

\begin{definition}[Maximal monotone operator]
  A monotone operator $A$ in $\reals^n$ is maximal monotone if there is no other monotone operator $B$ in $\reals^n$ such that (i) $A(u) \subseteq B(u)$ for all $u \in \reals^n$ and (ii) there exists at least one $u$ such that $B(u) \nsubseteq A(u)$.
\end{definition}

\subsection{Proof of Proposition~\ref{prop:admm_nonexp}}\label{proof:admm_nonexp}

We first introduce the following lemma.
\begin{lemma}\label{lemma:peaceman_rachford}
    Let $A$ and $B$ be maximal monotone operators and $R$ be a positive definite matrix.
    Then the operator $S = (2 J_{R^{-1} A} - I) \circ (2 J_{R^{-1} B} - I)$ is nonexpansive with respect to $R$.
\end{lemma}

\begin{proof}
    We first prove that $J_{R^{-1} C}$ is nonexpansive in $R$ for some maximal monotone operator $C$.
    First, fix $z$ and $z'$ both in $\reals^n$ and let $u = J_{R^{-1} C}(z)$ and $u' = J_{R^{-1} C}(z')$.
    By definition of the Resolvent, we have $R(u - z) \in C u$ and $R(u' - z') \in C u'$.
    By monotonicity, it follows that $(u-u')^T (R(u-z) - R(u'-z')) \geq 0$.
    Simplifying leads to $\|u-u'\|_R^2 \leq (u-u')^T R (z-z')$.
    It follows that $\|u-u'\|_R \leq \|z-z'\|_R$ by Cauchy-Schwarz.
    The proof concludes by noting that $2C - I$ is nonexpansive if $C$ is nonexpansive and that the composition of nonexpansive operators is nonexpansive.
\end{proof}

\myparagraph{ADMM}
Both $A = \eta \partial f$ and $B=\eta \partial g$ are maximal monotone since they are the subdifferentials of closed, convex, and proper functions.
The result follows directly from Lemma~\ref{lemma:peaceman_rachford}, where nonexpansiveness is proven with respect to $R=I$.

\myparagraph{OSQP}
First, we define the indicator function $\mathcal{I}_S(x) = 0$ if $x \in S$ and $\infty$ otherwise for any set $S$.
The iterates in Table~\ref{table:fp_algorithms} are a special case of ADMM, where $f(w,v) = (1/2)w^T P w + c^T w + \mathcal{I}_{\{A w = v\}}(w,v)$ and $g(w,v) = \mathcal{I}_{[l,u]}(v)$ and scaling $R = \diag(\sigma \ones, \rho)$; this is a straightforward extension of~\citep[Fact~4.1]{infeas_detection}.
Using maximal monotone operators $A = \partial f$ and $B = \partial g$, the proof concludes by using Lemma~\ref{lemma:peaceman_rachford}.

\myparagraph{SCS}
Let $F(z) = Mz + q$.
Solving the conic problem amounts to finding $z$ such that $0 \in F(z) + \mathcal{N}_\mathcal{C}(z)$~\citep[\Sec~3]{scs_quadratic} where $\mathcal{N}_\mathcal{C}=\{v \mid \sup_{v' \in \mathcal{C}} (v' - v)^T z \leq 0\}$ is the normal cone of $\mathcal{C} = \reals^p \times \mathcal{K}^*$.
The operator $F(z)$ is maximal monotone since $M+M^T \succeq 0$~\citep[Example~2.2.3]{lscomo}.
The operator $\mathcal{N}_\mathcal{C}(z)$ is maximal monotone since $\mathcal{N}_\mathcal{C}(z) = \partial \mathcal{I}_{\mathcal{C}}(z)$~\citep[\Sec~2.2]{lscomo}.
The first iterate in Table~\ref{table:fp_algorithms} is equivalent to $\tilde{u}^{k+1} = J_{R^{-1}F}(z^k)=(R+F)^{-1}(R z^k) = (R+M)^{-1}(R (z^k-q))$.
We claim that the second iterate in Table~\ref{table:fp_algorithms} is equivalent to $\tilde{u}^{k+1} = J_{R^{-1}\mathcal{N}_{\mathcal{C}}}(2\tilde{u}^{k+1}-z^k)=J_{\mathcal{N}_{\mathcal{C}}}(2\tilde{u}^{k+1}-z^k)$---\ie, the matrix $R$ does not affect the projection step~\citep{scs_code}.
To see this, observe that $l = (R + \mathcal{N}_{\mathcal{C}})^{-1} R (v) \implies 0 \in \mathcal{N}_{\mathcal{C}} (l) + R(l - v) \implies l = \argmin_z \|z - v\|_R^2 \mbox{ s.t. } z \in \mathcal{C}$.
This feasible set can be written as $\reals^p \times \reals^{\rm z} \times \mathcal{K}^*_{\rm nz}$ (where $\mathcal{K}^*_{\rm nz}$ denotes the dual of the part of $\mathcal{K}$ that is the non-zero cone).
Since the diagonal matrix $R$ is constant within each of these three cases, and the projection is separable across all $3$, the dependence on $R$ disappears.
The proof concludes by applying Lemma~\ref{lemma:peaceman_rachford}.

\section{Details of SDP from Section \ref{sec:robustness}}

\subsection{SDP for accelerated proximal gradient descent}\label{appendix:sdp_pgd}

\ifpreprint
\paragraph*{Convex case $(\mathcal{G} = \mathcal{F}_{\mu, L})$.} 
\else
\paragraph*{Convex case $(\mathcal{G} = \mathcal{F}_{\mu, L})$} 
\fi
Here we present the details on deriving \eqref{prob:sdp_pgd} from \eqref{prob:pep_pgd}. Let $G = P^T P$, where
\[P = \begin{pmatrix}
    z^0 & z^{\star} & \nabla f(z^0) & \dots & \nabla f(z^K) & \nabla f(z^\star) & \partial g(z^0) & \dots & \partial g(z^K) & \partial g(z^\star)
\end{pmatrix}\,.\]
There exist vectors $\rho^0, \dots, \rho^K, \rho^\star$ (depending on $\theta$) and $\sigma^0, \dots, \sigma^K, \sigma^\star$, $\tau^0, \dots, \tau^K, \tau^\star$  such that $z^k = P\rho^k, \nabla f(z^k) = P\sigma^k, \partial g(z^k) = P\tau^k$ for all $k = 0, \dots, K, \star$.

Now define $f^k = f(z^k)$ and $g^k = g(z^k)$.  For $f \in \mathcal{F}_{\mu, L}$ and $g \in \mathcal{F}_{0, \infty}$, the corresponding interpolation inequalities \citep{taylor2017exactworst, pep2} are 
\begin{align*}
    f^j &\geq f^i + \Tr(G(\rho^j - \rho^i)\otimes \sigma^i)  \\
    &\quad + \frac{1}{2L} \Tr(G(\sigma^i - \sigma^j)^{\otimes 2}) + \frac{\mu}{2(1-\mu/L)}\Tr(G(\rho^i - \rho^j - \frac{1}{L}(\sigma^i - \sigma^j))^{\otimes 2}), \quad \forall i, j, \\
    g^j &\geq g^i + \Tr(G(\rho^j - \rho^i) \otimes \tau^i), \quad \forall i, j,
\end{align*}
where $a \otimes b$ denotes the outer product between vectors $a, b$ and $a^{\otimes 2} = a \otimes a$. These can be written as in \eqref{prob:sdp_pgd} for appropriate choices of $B^{ij}, C^{ij}$. 

The remaining constraints of \eqref{prob:pep_pgd} can be written as in \eqref{prob:sdp_pgd} by letting $A^0 = (e^1 - e^2)^{\otimes 2}$ and $A^* = (e^{K+4} + e^{2K+6})^{\otimes 2}$, where $e^i$ is the $i$-th standard basis vector. For the performance metric, for $r(z) = f(z) +g(z)-f(z^\star) - g(z^\star)$ we let $U = 0, v^i = w^i = \bm{1}\{i = K\} - \bm{1}\{i = \star\}$; for $r(z) = \|z - z^\star \|^2$ we let $U = (\rho^K - \rho^\star)^{\otimes 2}, v^i = w^i \equiv 0$. These respectively correspond to Cases 1 and 2 in Assumption \ref{assumption:proxgd}.

\ifpreprint
\paragraph*{Quadratic case $(\mathcal{G} = \mathcal{Q}_{\mu, L})$.} 
\else
\paragraph*{Quadratic case $(\mathcal{G} = \mathcal{Q}_{\mu, L})$} 
\fi
First, we write the SDP formulation in this case:
\begin{equation}\label{prob:sdp_pgd_quad}
  \begin{array}{lll}
\mbox{maximize} & \mbox{(performance metric)} & \Tr(G U) \\
  \mbox{subject to} & \mbox{(initial point)} & \Tr(G A^0) \leq 1 \\
  & \mbox{(optimality)} & \Tr(GA^\star) = 0 \\
  & \mbox{(algorithm update} & B^T G B' = (B')^T G B, C^T G C' \succeq 0 \\
  & \mbox{ + function class)} & g^j \geq g^i + \Tr(G D^{ij}),\quad \forall i, j \\
  & \mbox{(Gram matrix)} & G \succeq 0.
  \end{array}
\end{equation}

For derivation, we start by assuming that (by adding linear terms to $g$ if necessary) that $f$ is homogeneous quadratic. Using the same notations of $G, P, \rho^k, \sigma^k, \tau^k$ as in the previous case, we further denote $\rho = \begin{pmatrix} \rho^0 & \dots & \rho^K & \rho^\star \end{pmatrix}$ and $\sigma = \begin{pmatrix} \sigma^0 & \dots & \sigma^K & \sigma^\star \end{pmatrix}$ so that $P\rho = \begin{pmatrix}
    z^0 & \dots & z^K & z^\star
\end{pmatrix}$ and $P\sigma = \begin{pmatrix}
    \nabla f(z^0) & \dots & \nabla f(z^K) & \nabla f(z^\star)
\end{pmatrix}$.

Then the corresponding interpolation inequalities \citep{pep2, bousselmi2024interpolation} are $\rho^T G \sigma = \sigma^T G \rho, (\sigma - \mu \rho)^T G (L\rho - \sigma) \succeq 0, g^j \geq g^i + \Tr(G(\rho^j - \rho^i) \otimes \tau^i) \ \forall i, j$,
which specify $B, B', C, C', D^{ij}$. The constraints on initial point and optimality follow from those as in $\mathcal{G} = \mathcal{F}_{\mu, L}$ (\ie, the same choices of $A^0$ and $A^\star$). 
Finally, for the performance metric, $U = (\rho^K - \rho^\star)^{\otimes 2}$ (recall that this is for Case 3 in Assumption \ref{assumption:proxgd}).

\subsection{SDP for accelerated ADMM}\label{appendix:sdp_admm}
In this subsection, we derive \eqref{prob:sdp_admm} from \eqref{prob:pep_admm}. Similar to Appendix Subsection \ref{appendix:sdp_pgd}, we let $G = P^TP$ where %
\[P = \begin{pmatrix}
    R^{1/2}z^0 & R^{1/2}z^{\star} & R^{1/2}S(z^0) & \dots & R^{1/2}S(z^K) & R^{1/2}S(z^\star)
\end{pmatrix}\,.\]
Then there exist vectors $\rho^k, \sigma^k$ for $k=0,\dots,K,\star$ (depending on $\theta$, distinct from those in the last section) such that $R^{1/2}z^k = P\rho^k, R^{1/2}S(z^k) = P\sigma^k$.
The corresponding interpolation inequalities of the nonexpansiveness of $S$ with respect to $R$~\citep{ryu_ospep} are $\Tr(G (\sigma^i - \sigma^j)^{\otimes 2}) \leq \Tr(G (\rho^i - \rho^j)^{\otimes 2}) \ \forall i, j$,
which can be written as in \eqref{prob:sdp_admm} with $B^{ij} = (\sigma^i - \sigma^j)^{\otimes 2} - (\rho^i - \rho^j)^{\otimes 2}$. For the remaining constraints of \eqref{prob:pep_admm}, we take $A^0 = (e^1 - e^2)^{\otimes 2}$ and $A^\star = (e^2 - e^{K+4})^{\otimes 2}$ in \eqref{prob:sdp_admm}. Also, $U = (\rho^K - \sigma^K)^{\otimes 2}$.

\end{document}